\documentclass[final]{m2an-HACK}
\usepackage{amssymb,amsmath,amsfonts,amsthm}
\usepackage[UKenglish]{babel}
\usepackage[utf8]{inputenc}
\usepackage[T1]{fontenc}
\usepackage{verbatim}
\usepackage{graphicx}
\usepackage{xcolor}
\usepackage{enumerate}
\usepackage[numbers]{natbib}
\usepackage{algorithm}
\usepackage{algpseudocode}
\newcommand\revred{}
\newcommand\revblue{}
\numberwithin{equation}{section}
\newtheorem{theorem}{Theorem}[section]
\newtheorem{proposition}[theorem]{Proposition}
\newtheorem{corollary}[theorem]{Corollary}
\newtheorem{lemma}[theorem]{Lemma}
\newtheorem{remark}[theorem]{Remark}
\renewcommand{\paragraph}[1]{\medskip\noindent \textbf{#1.} \,}
\newcommand{\s}{\sigma}
\renewcommand{\t}{\tau}
\newcommand{\wt}[1]{\widetilde{#1}}
\newcommand{\NN}{\mathbb{N}}
\newcommand{\RR}{\mathbb{R}}
\newcommand{\CC}{\mathbb{C}}
\newcommand{\dd}{\mathrm{d}}
\newcommand{\ee}{\mathrm{e}}
\newcommand{\ii}{\mathrm{i}}
\newcommand{\Ah}{A_h}
\newcommand{\Bh}{B_h}
\newcommand{\Bt}{\widehat{B}}
\newcommand{\Bth}{\Bt_h}
\newcommand{\DP}{\Delta^{-1}}
\newcommand{\DPh}{\Delta^{-1}_h}
\newcommand{\nC}{\mathcal{C}}

\newcommand{\EF}{\varphi_{\text{\tiny{SP}}}}
\newcommand{\EA}{\varphi_A}
\newcommand{\EAh}{\varphi_{A_h}}
\newcommand{\EB}{\varphi_B}
\newcommand{\EBh}{\varphi_{B_h}}
\newcommand\eA{\EA(\half \t)}
\newcommand\eAh{\EAh(\half \t)}

\newcommand\eb{\mathcal{E}_B}
\newcommand\ebh{\mathcal{E}_{B_h}}
\newcommand\ebt[1]{\mathcal{E}_B(\t,{#1})}

\newcommand\ebht[1]{\mathcal{E}_{B_h}(\t,{#1})}
\newcommand\ebhtee[1]{\mathcal{E}_{B_h}(t,{#1})}
\newcommand{\nS}{\mathcal{S}}
\newcommand{\nnS}[1]{\mathcal{S}^{#1}\,}
\newcommand{\nnSh}[1]{\mathcal{S}_h^{#1}\,}
\newcommand{\Ih}{\mathcal{I}_h}
\newcommand{\Ph}{\mathcal{P}_h}
\newcommand{\nV}{\mathcal{V}}
\newcommand{\nVh}{\nV^h}
\newcommand\wu{\wt{u}}
\newcommand{\ini}{\psi_0}

\newcommand{\sgn}{\mathrm{sgn}}
\newcommand\half{\tfrac{1}{2}}

\newcommand{\scp}[3]{(#2,#3)_{#1}}

\newcommand{\norm}[2]{\|#2\|_{#1}}
\newcommand\ol[1]{\overline{#1}}
\newcommand{\const}{{\mathcal C}}
\newcommand{\nO}{\mathcal{O}}

\newcommand{\nP}{\mathcal{P}}

\newcommand{\pot}{\Theta}
\newcommand{\potm}{\Phi}

\begin{document}
\title{Convergence of a Strang splitting finite element discretization for the Schr{\"o}dinger-Poisson equation}%
\thanks{This work was supported by the Austrian Science Fund (FWF) under grants P24157-N13 and P21620-N13, the Vienna Science and Technology Fund (WWTF) under the grant MA14-002 and the Doktoratsstipendium of the University of Innsbruck.}
\author{Winfried~Auzinger}
\address{Technische Universit{\"a}t Wien, Institut für Analysis und Scientific Computing, Wiedner Hauptstra{\ss}e \mbox{8-10}, A-1040 Wien, Austria.} \author{Thomas~Kassebacher}
\address{Leopold-Franzens Universit{\"a}t Innsbruck, Institut f{\"u}r Mathematik, Technikerstra{\ss}e 13, A-6020 Innsbruck, Austria.}
\author{Othmar~Koch}
\address{Universit{\"a}t Wien, Fakult{\"a}t f{\"u}r Mathematik, Oskar-Morgenstern-Platz~1, A-1090 Wien, Austria.}
\author{Mechthild~Thalhammer}
\sameaddress{2}
\begin{abstract}
Operator splitting methods combined with finite element spatial discretizations are studied for time-dependent nonlinear  Schrödinger equations.
In particular, the Schrödinger--Poisson equation under homogeneous Dirichlet boundary conditions on a finite domain is considered.
A rigorous stability and error analysis is carried out for the second-order Strang splitting method and conforming polynomial finite element discretizations.
For sufficiently regular solutions the classical orders of convergence are retained, that is,
second-order convergence in time and polynomial convergence in space is proven.
The established convergence result is confirmed and complemented by numerical illustrations.
\end{abstract}
\maketitle

\section{Introduction and overview} \label{sec:intro}
We consider full discretization methods for the time-dependent Schr{\"o}dinger--Poisson equation,  which typically arises in models of quantum transport~\cite{bremar91,illneretal94}.
Our approach relies on a second-order Strang splitting time discretization combined with a conforming $ hp $ finite element space discretization.
The motivation for the proposed solution method is that separate treatment of the nonlinear part suggests the application of special solvers for the Poisson equation, which are particularly efficient in the context of an underlying finite element space discretization.
For this purpose it is common to truncate the unbounded spatial domain to a sufficiently large finite domain
and impose homogeneous Dirichlet boundary conditions.
{\revred Indeed, the evaluation of the nonlocal convolution integral in the standard formulation generally implies
a huge computational effort caused by the suitable treatment of the singular integral
kernel for the evaluation on a large domain. By the splitting approach, we can separately treat the Poisson equation by appropriate methods where optimized linear solvers are available as for instance multigrid
or domain decomposition methods~\cite{olstyr14,saad03}. The finite element discretization additionally
enables a solution on a solution-adapted non-uniform spatial grid, which can be updated in the course
of the time integration~\cite{schiesser91}.}

Our main objective is to provide an error analysis for this full discretization, showing the expected second-order convergence of the Strang splitting method and polynomial spatial error decay corresponding with the finite elements employed.
By using Gauss--Lobatto nodes, the setup of the stiffness matrix is exact;
the errors arising in the construction of the mass matrix and the right-hand side are of higher order than the discretization error and therefore will not be taken into account.

\paragraph{Splitting methods}
The computational advantages of operator splitting methods for the time integration of problems in quantum dynamics have been emphasized in recent literature.
A comprehensive overview of investigations for time-dependent Gross--Pitaevskii equations is given in~\cite{baoetal13a}, which summarizes most of the studies conducted in this field.
The Crank-Nicholson finite difference method preserves most of the important invariants like symmetry in time, mass and energy and is unconditionally stable;
however, the computational cost for this fully implicit method is considerable, and the conservation properties only hold up to the accuracy of the nonlinear solver.
Semi-implicit relaxation methods which only treat the kinetic part implicitly share the conservation properties if only a cubic nonlinearity is present, but they are still computationally expensive and suffer from stability limitations.
Semi-implicit finite difference schemes lose most of the desired properties.
For regular solutions, time-splitting methods in conjunction with Fourier- or Sine-spectral methods are overall concluded to be the most successful discretization schemes;
they are unconditionally stable, conserve norm, energy, and also dispersion, which is not the case for many other time-stepping schemes.
For non-smooth or random spatial profiles, the spectral accuracy may be lost, however, and thus splitting methods in conjunction with finite difference spatial discretizations may be more efficient (see~\cite{baobecreview13}).

Recently, full discretization of the Schr{\"o}dinger--Poisson equation by splitting methods in conjunction with spectral space discretization has been investigated in~\cite{bao15speq}, where the long-range interaction is approximated efficiently by nonuniform fast Fourier transform (NUFFT).
The authors conclude superior accuracy and performance of their approach in particular over the Sine-spectral method.

\paragraph{Error analysis}
The stability and error behavior of operator splitting methods for the Schr{\"o}dinger--Poisson equation have first been analyzed in~\cite{lubich07}.
For the structurally similar equations associated with the multi-configuration time-dependent Hartree--Fock method, a complete convergence analysis of high-order splitting methods has been given in~\cite{knth10a}.
An error analysis of splitting methods applied to the Schr{\"o}dinger--Poisson equation in the semiclassical regime is provided in~\cite{carles13}.

\paragraph{Finite element method}
The literature on finite element spatial discretizations is vast.
Finite element methods (FEM) have been widely used for electronic structure calculations, see for instance~\cite{gargri00,schauer14} and~\cite{baodft12,fangetal12,motamarrietal12,zhouetal06}. For the solution of time-dependent Schr{\"o}dinger equations see for example~\cite{bandrauk95,makridakis96} and the more recent contribution~\cite{katkyz15}, and for atomic and molecular systems see~\cite{HocBon2014} for a general review.

\paragraph{Truncation to a finite domain}
In conjunction with the application of the finite element method, the restriction to a finite domain introduces a truncation error which we do not consider in this work.
Strategies to cope with related issues have been proposed for instance in~\cite{antoinebesseklein12}.
The investigation in the context of the Schr{\"o}dinger--Poisson equation remains an open question.

\paragraph{Outline}
{In Sec.~\ref{sec:setting} we state the Schr{\"o}dinger--Poisson equation.
We specify the full discretization method and formulate our
main convergence results.
In Sec.~\ref{sec:conv-ana} we provide
the underlying comprehensive stability and error analysis.
Our numerical illustrations given in Sec.~\ref{sec:num} confirm the theoretical convergence result and demonstrate that also higher-order splitting methods show their expected behavior.
The appendices contain proof details, important results from the literature
which we rely on and auxiliary estimates used in our analysis.
}

\section{Problem setting, discretization method,
         {\revred and main results}} \label{sec:setting}
\subsection{Problem setting}
\paragraph{Schr{\"o}dinger-Poisson equation}
We consider the time-dependent Schr{\"o}dinger-Poisson equation for
$\psi\colon \Omega \times [0, T] \rightarrow \CC,\,(x,t) \mapsto \psi(x,t)$,
\begin{subequations}
\label{eq:SPE}
\begin{equation}
\ii\,\partial_t \psi(x,t) = - \tfrac {1}{2}\,\Delta \psi(x,t) + \DP(|\psi(x,t)|^2)\,\psi(x,t)\,, 
\end{equation}
where $\Omega \subset \RR^d$, $d \in \{2,3\}$,
is a bounded domain with smooth boundary.

We impose homogeneous Dirichlet boundary conditions and an initial condition
\begin{equation}
\psi(x,t)\big|_{x\in \partial \Omega} = 0\,, \quad \psi(x,0) = \ini(x)\,.
\end{equation}
For the subsequent analysis we will assume that the initial state
satisfies\footnote{For simplicity of notation we write $ L^2, H^k $ instead of
                   $ L^2(\Omega), H^k(\Omega) $, etc.}
$\ini \in H^2 = H^2(\Omega) $.
The nonlocal nonlinear term $\DP(|\psi|^2)$ describing the electrostatic self-interaction
is the solution $ \pot $ of the Poisson equation under homogeneous Dirichlet boundary conditions,
\begin{equation}
\Delta \pot(x,t) = |\psi(x,t)|^2\,, \quad \pot(x,t)\big|_{x\in \partial \Omega} = 0\,.
\end{equation}
\end{subequations}
The evolution operator associated with problem~\eqref{eq:SPE} will be denoted by $ \EF $, i.e.,
\begin{equation*}
\psi(\,\cdot\,,t) = \EF(t,\psi_0)\,.
\end{equation*}
\paragraph{Abstract formulation}
Introducing the operator notation
\begin{subequations}
\label{eq:operator_notation}
\begin{equation}
\label{eq:AB}
\begin{split}
&A\colon\, H^2 \cap H^1_0 \to L^2\colon~~ u \,\mapsto\,\tfrac{1}{2}\,\ii\,\Delta u\,, \\
&\Bt\colon\, H^1_0 \to H^2 \cap H^1_0\colon~~ w \,\mapsto\, -\,\ii\,\DP(|w|^2)\,, \\
&B\colon\, H^1_0 \to H^1_0\colon~~~~~~~~\; u \,\mapsto\, -\,\ii\,\DP(|u|^2)\,u\,,
\end{split}
\end{equation}
we employ a compact formulation of problem~\eqref{eq:SPE} as an abstract evolution equation
\begin{equation}
\label{eq:SPEproblem}
\begin{cases}
~\partial_t \psi =A\,\psi + B(\psi) = A\,\psi + \Bt(\psi)\psi\,, \\
~\psi\big|_{t=0} =\ini\,.
\end{cases}
\end{equation}
\end{subequations}
\subsection{Semidiscretization in time by the Strang splitting method}
\paragraph{Subproblems}
For the discretization of~\eqref{eq:SPEproblem} in time
we apply exponential operator splitting methods based on the solution of two subproblems,
see for instance~\cite{haireretal02,macqui02}.
\begin{itemize}
\item
The evolution operator associated with the linear initial value problem
\begin{subequations} \label{eq:EA-all}
\begin{equation}
\label{eq:EA-problem}
\begin{cases}
~\partial_t \psi = A\,\psi\,, \\
~\psi\big|_{t=0} = u\,,
\end{cases}
\end{equation}
is denoted by $ \EA(t) $, such that
\begin{equation}
\label{eq:EA-def}
\psi(\,\cdot\,,t) = \EA(t)\,u\,.
\end{equation}
\end{subequations}
\item
The evolution operator associated with the nonlinear initial value problem
\begin{subequations} \label{eq:EB-all}
\begin{equation}
\label{eq:EB-problem}
\begin{cases}
~\partial_t \psi = B(\psi)\,, \\
~\psi\big|_{t=0} = u\,,
\end{cases}
\end{equation}
is denoted by $ \EB(t,\,\cdot\,) $, such that
\begin{equation}
\label{eq:EB-def}
\psi(\,\cdot\,,t) = \EB(t,u)\,.
\end{equation}
Due to the fact that $\DP(|\psi(\,\cdot\,,t)|^2)$ defines a real-valued function
and thus
\begin{equation*}
\partial_t |\psi(\,\cdot\,,t)|^2
= 2\,\Re\big(\overline{\psi(\,\cdot\,,t)}\,\partial_t \psi(\,\cdot\,,t)\big) = 2\,\Re\big(\Bt(\psi(\,\cdot\,,t))\,|\psi(\,\cdot\,,t)|^2\big) = 0\,,
\end{equation*}
the nonlinear equation~\eqref{eq:EB-problem} reduces to the linear equation
\begin{equation}
\label{eq:EB-problem-linear}
\begin{cases}
~\partial_t \psi = \Bt(u)\psi\,,  \\
~\psi\big|_{t=0} = u\,.
\end{cases}
\end{equation}
\end{subequations}
We will also employ a notation analogous to~\eqref{eq:EB-problem-linear}
but with a linear evolution operator $ \eb $ depending on $ u $ and $w$ as the solution to
\begin{equation}
\label{eq:eb-problem-linear}
\begin{cases}
~\partial_t \psi = \Bt(w)\psi\,,  \\
~\psi\big|_{t=0} = u\,,
\end{cases}
\end{equation}
such that $\psi=\eb(t,w)\,u$. Clearly,
\begin{equation}
\label{eq:EB}
\EB(t,u)=\eb(t,u)\, u \,.
\end{equation}
\end{itemize}

\paragraph{Strang splitting method}
Our main focus is on the symmetric second-order Strang splitting method applied to the splitting
according to~\eqref{eq:EA-all}, \eqref{eq:EB-all}.
That is, for a time increment $\tau > 0$, the time-discrete solution values
\begin{equation*}
\psi_n \approx \psi(n \t)\,, \quad n=0,1,2,\ldots
\end{equation*}
are determined by the recurrence
\begin{subequations}
\label{eq:Strang-splitting}
\begin{equation}
\label{eq:Strang-splitting-1}
\psi_n = \nS(\t,\psi_{n-1}) = \EA(\half \t)\,\EB\big(\t,\EA(\half \t) \psi_{n-1}\big)\,.
\end{equation}
For notational simplicity we shall employ a formal notation for the $n$-fold composition,
\begin{equation}
\label{eq:Strang-splitting-2}
\psi_n = \nnS{n}\ini := \underbrace{\nS(\t,\,\cdot\,) \circ \cdots \circ \nS(\t,\ini)}_{n \text{ times}}\,.
\end{equation}
\end{subequations}

\paragraph{Weak formulation of the subproblems}
In view of full discretization (see Sec.~\ref{sec:conforming-FEM}) we consider the following weak formulations of the subproblems.
For~\eqref{eq:EA-problem},
\begin{equation}
\label{eq:SubproblemWeak-A}
\begin{cases}
~\scp{L^2}{\partial_t \psi}{\phi} = -\,\tfrac{1}{2}\,\ii\,\scp{L^2}{\nabla\psi}{\nabla\phi}
 \quad \text{for all} ~~ \phi \in H^1_0\,,  \\
~\psi\big|_{t=0} = u\,,
\end{cases}
\end{equation}
where we require $\psi, u \in H^1_0$.
\begin{subequations}
\label{eq:SubproblemWeak-B}
For~\eqref{eq:EB-problem-linear},
\begin{equation}
\label{eq:EB-weak-problem}
\begin{cases}
~\scp{L^2}{\partial_t \psi}{\phi} = -\,\ii\,\scp{L^2}{\pot\,\psi}{\phi}
  \quad \text{for all} ~~ \phi \in H^1_0\,,  \\
~\psi\big|_{t=0} = u\,,
\end{cases}
\end{equation}
where $\pot$ is the solution of the Poisson equation in weak formulation,
\begin{equation}
\label{eq:EB-poisson-problem}
\scp{L^2}{\nabla\pot}{\nabla\chi} = -\,\scp{L^2}{|u|^2}{\chi}
\quad \text{for all} ~~ \chi \in H^1_0\,,
\end{equation}
\end{subequations}
requiring $\psi,\pot, u \in H^1_0$.

{\revred In the following we use the standard denotation
for the Sobolev semi-norms,
i.e., $|\psi|_{H^1} = \| \nabla\psi \|_{L^2} $
for $ \psi \in H_0^1 $, and
$ |\psi|_{H^2} = \big( \sum_{|\alpha|=2}
  \|D^\alpha \psi\|_{L^2}^2 \big)^{1/2} $
for $ \psi \in H^2 $.}

\subsection{Conforming finite element discretization of the subproblems} \label{sec:conforming-FEM}
A full discretization arises by solving both initial value subproblems~\eqref{eq:EA-all}
and~\eqref{eq:EB-all} in their weak reformulation~\eqref{eq:SubproblemWeak-A} and~\eqref{eq:SubproblemWeak-B},
respectively, by means of a finite element method (FEM).

\paragraph{Finite element space}
For the space discretization of the subproblems, we choose a tessellation~$\mathcal{T}^h$
over subdomains $ \Omega_k $, with
\begin{equation*}
\Omega = \bigcup_{k=1}^{K}\,\Omega_k\,, \qquad
h = \max_{k\,\in\,\{1,\dots,K\}} \mathrm{diam}\,\Omega_k\,,
\end{equation*}
which are affine-equivalent to a reference domain~$\Omega_0$.
With~$\Omega_0$ we associate a triplet $(\Omega_0,P,\mathcal{N})$,
where the set~$\mathcal{N}$ comprises the interpolation nodes~$x_i$,
and~$P$ is the linear space spanned by the
polynomial nodal basis functions $ v_j $ of degree~$p$.
We require the finite elements to be conforming and quasi-uniform.
As common we choose a linear indexing of the basis functions, $(v_j)_{j=1}^{J}$.
The subspace spanned by these functions is denoted by
\begin{equation}
\label{eq:Vbasis-vj}
\nVh = \mathrm{span} \{v_1, \dots, v_J\} \subset H^1_0\,.
\end{equation}

\paragraph{Finite element interpolation and projection}
By $\Ih\colon\,\nC(\Omega) \to \nVh$ we denote the nodal interpolation operator,
\begin{equation}
\label{eq:Ih-def}
\Ih(f) = \sum_{j=1}^{J} f(x_j)\,v_j\,.
\end{equation}
The Rayleigh-Ritz projection $\Ph\colon\, H_0^1 \to \nVh$ is defined implicitly by
the Galerkin orthogonality relation
\begin{subequations}
\begin{equation}
\label{eq:Ph-def}
\scp{L^2}{\nabla (u - \Ph\,u)}{\nabla v_h} = 0 \quad \text{for all} ~~ v_h\,\in\,\nVh\,,
\end{equation}
satisfying
\begin{equation}
|\Ph\,u|_{H^1} \leq |u|_{H^1} \quad \text{for all} ~~ u \in H_{0}^1\,.
\end{equation}
By the Poincar{\'e} inequality {\revred $\|v\|_{L^2}\leq C\,|v|_{H^1}$}, this also implies
\begin{equation}
\label{eq:Proj-prop}
\| \Ph\,u \|_{H^1} \leq C\,\| u \|_{H^1} \quad \text{for all} ~~ u \in H_{0}^1\,,
\end{equation}
with a constant $ C $ depending on $\Omega$.
\end{subequations}

\begin{remark} \label{rem:Ph}
{\rm{
The Rayleigh-Ritz projection $ \Ph $ is connected to {\revred the} finite element approximation in
the following way.
Consider a Poisson problem $ \Delta u = f $ with homogeneous Dirichlet boundary conditions in weak formulation
(see~\eqref{eq:EB-poisson-problem}),
\begin{equation*}
\scp{L^2}{\nabla u}{\nabla v} = -\,\scp{L^2}{f}{v}
\quad \text{for all} ~~ v \in H^1_0\,,
\end{equation*}
and its FEM discretization by the Galerkin equations (see~\eqref{eq:EBh-poisson-problem} below),
\begin{equation*}
\scp{L^2}{\nabla u_h}{\nabla v_h} = -\,\scp{L^2}{f}{v_h}
\quad \text{for all} ~~ v_h \in \nVh\,.
\end{equation*}
Then,
\begin{equation*}
\scp{L^2}{\nabla (u - u_h)}{\nabla v_h} = 0 \quad \text{for all} ~~ v_h\,\in\,\nVh\,,
\end{equation*}
i.e.,
\begin{equation*}
u_h = \Ph\,u\,.
\end{equation*}
}}
\end{remark}
{\revred For a sufficiently smooth boundary,
the $ H^2 $ regularity estimate
\begin{equation}
\label{eq:Omega-property}
|\DP(u-\Ph u)|_{H^2}  \leq C\, \| u-\Ph u\|_{L^2}
\end{equation}
holds, see~\cite[Sec.~5.5]{brennerscott02}.}

\paragraph{Fully discrete solution and computational representation}
The full discretization of the Schr{\"o}dinger-Poisson equation is based on
solving the subproblems~\eqref{eq:SubproblemWeak-A}, \eqref{eq:SubproblemWeak-B}
arising in the Strang splitting time discretization by means of a FEM/Galerkin space discretization.
Here, the coefficients associated with the prescribed initial state are
determined by interpolation,
\begin{equation*}
\Ih(\ini) = \sum_{j=1}^{J} c_{j,0}\,v_j \in \nVh\,.
\end{equation*}
In each substep of the time propagation by Strang splitting,
subproblems of the following types~\eqref{eq:DiscreteSubproblemWeak-A}, \eqref{eq:DiscreteSubproblemWeak-B}
for the solutions $\psi_h\in \nVh$ and $\pot_h\in \nVh$ are solved.
\begin{itemize}
\item
\begin{subequations}
\label{eq:DiscreteSubproblemWeak-A}
For the first subproblem~\eqref{eq:SubproblemWeak-A}, $\psi_h$ is determined from
\begin{equation}
\label{eq:EAh-weak-problem}
\begin{cases}
~\scp{L^2}{\partial_t \psi_h}{\phi_h} = -\,\tfrac{1}{2}\,\ii\,\scp{L^2}{\nabla\psi_h}{\nabla\phi_h}
\quad \text{for all} ~~ \phi_h \in \nVh\,,  \\
~\psi_h\big|_{t=0} = u_h\,.
\end{cases}
\end{equation}
With the ansatz in terms of the basis~\eqref{eq:Vbasis-vj},
\begin{equation*}
\psi_h(t) = \sum_{j=1}^{J} c_j(t)\,v_j\,,
\end{equation*}
\eqref{eq:EAh-weak-problem} yields the Galerkin equations for the coefficients $ c_j(t) $ in the form
\begin{equation}
\label{eq:EAh-weak-problem-coeff}
\sum_{j=1}^{J} \partial_t\,c_j(t)\,\scp{L^2}{v_j}{v_i} = -\,\tfrac{1}{2}\,\ii\,\sum_{j=1}^{J} c_j(t)\,\scp{L^2}{\nabla v_j}{\nabla v_i}\,,
\quad i=1 \ldots J\,.
\end{equation}
\end{subequations}
\item
For the second subproblem~\eqref{eq:SubproblemWeak-B}, $ \psi_h $ is determined such that
\begin{subequations}
\label{eq:DiscreteSubproblemWeak-B}
\begin{equation}
\label{eq:EBh-weak-problem}
\begin{cases}
~\scp{L^2}{\partial_t \psi_h}{\phi_h} = -\,\ii\,\scp{L^2}{\pot_h\,\psi_h}{\phi_h}
\quad \text{for all} ~~ \phi_h \in \nVh\,,  \\
~\psi_h\big|_{t=0} = u_h\,,
\end{cases}
\end{equation}
where $\pot_h$ is the solution of the discretized Poisson problem
\begin{equation}
\label{eq:EBh-poisson-problem}
\scp{L^2}{\nabla\pot_h}{\nabla\chi_h} = -\,\scp{L^2}{|u_h|^2}{\chi_h}
\quad \text{for all} ~~ \chi_h \in \nVh\,.
\end{equation}
With the ansatz in terms of the basis~\eqref{eq:Vbasis-vj},
\begin{equation*}
\psi_h(t) = \sum_{j=1}^{J} c_j(t)\,v_j \in \nVh\,, \qquad \pot_h = \sum_{j=1}^{J} d_j\,v_j \in \nVh\,,
\end{equation*}
we obtain the Galerkin equations~\eqref{eq:EBh-poisson-problem} in the form
\begin{equation}
\label{eq:EBh-weak-problem-coeff}
\begin{split}
\sum_{j=1}^{J} \partial_t\,c_j(t)\,\scp{L^2}{v_j}{v_i}
&= -\,\ii\,\sum_{j=1}^{J} c_j(t)\,\sum_{k=1}^{J} d_k\,\scp{L^2}{ v_k\, v_j}{v_i}, \quad i=1 \dots J\,, \\
\sum_{j=1}^{J} d_j\,\scp{L^2}{\nabla v_j}{\nabla v_i}
&= -\!\sum_{j,k=1}^{J} c_j(0)\,\overline{c_k(0)}\,\scp{L^2}{v_j\,v_k}{v_i}\,, \quad i=1 \dots J\,.
\end{split}
\end{equation}
\end{subequations}
\end{itemize}
{\revred The above computations imply that the unknown coefficients
$c_j,\ d_j$ satisfy systems
of linear ordinary differential equations.}
For a more compact formulation we introduce the vectors
\begin{subequations}
\label{eq:Matrix-definition}
\begin{equation}
\begin{gathered}
c(t) = \big( c_j(t) \big)_{j=1}^{J} \in \CC^{J}\,, \qquad d = \big( d_j \big)_{j=1}^{J} \in \RR^{J}\,, \\
F(c(0)) = \Big(\sum_{j,k=1}^{J} c_j(0)\,\overline{c_k(0)}\,\scp{L^2}{v_j\,v_k}{v_i}\Big)_{i=1}^{J} \in \RR^{J}\,,
\end{gathered}
\end{equation}
and the invertible symmetric matrices
\begin{equation}
\begin{aligned}
M &= \big( M_{ij} \big)_{i,j=1}^{J} \in \RR^{J \times J}\,, \quad \text{with} ~~ M_{ij} = \scp{L^2}{v_i}{v_j}\,, \\
K &= \big( K_{ij} \big)_{i,j=1}^{J} \in \RR^{J \times J}\,, \quad \text{with} ~~ K_{ij} = \scp{L^2}{\nabla v_i}{\nabla v_j}\,, \\
\potm(d) &= \big( \potm_{ij}(d) \big)_{i,j=1}^{J} \in \RR^{J \times J}\,,
\quad \text{with} ~~ \potm_{ij}(d) =\sum_k d_k \scp{L^2}{v_k\,v_i}{v_j}\,.
\end{aligned}
\end{equation}
\end{subequations}
In this notation, the system~\eqref{eq:EAh-weak-problem-coeff} reads
\begin{equation}
\label{eq:EAh-problem}
\begin{cases}
~M\,\partial_t\,c(t) &= -\,\tfrac{1}{2}\,\ii\,K c(t)\,, \\
~c(0) & \text{given}\,,
\end{cases}
\end{equation}
with solution
\begin{equation}
c(t) = \ee^{-\,\frac{1}{2}\,\ii\,t\,M^{-1} K}\,c(0)\,.
\end{equation}
System~\eqref{eq:EBh-weak-problem-coeff} takes the form
\begin{equation}
\label{eq:EBh-problem}
\begin{cases}
~K d& = -\,F(c(0)) \,, \\
~M\,\partial_t\,c(t) &= -\,\ii\,\potm(d)\, c(t)\,, \\
~c(0) & \text{given}\,,
\end{cases}
\end{equation}
with solution
\begin{equation}
c(t) = \ee^{-\,\ii\,t\,(M^{-1} \potm(d))}\,c(0)\,, \text{ where}~~d=-K^{-1}F(c(0))\,.
\end{equation}
To realize the fully discrete propagation in time
according to the Strang recurrence~\eqref{eq:Strang-splitting}, systems of this type are alternately solved.

\paragraph{Finite element operators}
We define the discrete Laplace operator
$\Delta_h\colon\,\nVh \to H^{-1}$ ($ u_h \mapsto \Delta_h\,u_h $)
and its inverse $\Delta_h^{-1}\colon\,H^{-1} \to \nVh$ ($ f \mapsto \Delta_h^{-1} f $) via
\begin{subequations}
\label{eq:Deltah+i-Def}
\begin{align}
& \scp{L^2}{\Delta_h\,u_h}{v_h} = -\,\scp{L^2}{\nabla u_h}{\nabla v_h} \quad \text{for all} ~~ v_h \in \nV^h\,, \label{eq:Deltah-Def} \\
& \scp{L^2}{\nabla\,\Delta_h^{-1} f}{\nabla v_h} = -\,\scp{L^2}{f}{v_h} \quad~\, \text{for all} ~~ v_h \in \nV^h\,. \label{eq:Deltahi-Def}
\end{align}
\end{subequations}
In particular, \eqref{eq:Deltahi-Def} means that for the solution $ u $ of $ \Delta\,u = f $ we have
\begin{equation}
\label{eq:Ph-Deltahi}
\Delta_h^{-1} f = \nP_h\,u = u_h
\end{equation}
in the sense of Remark~\ref{rem:Ph}.

Moreover, in analogy to~\eqref{eq:AB} we set
\begin{equation}
\label{eq:AhBh}
\begin{split}
&\Ah\colon\,\nVh \to H^{-1}\colon~~ u_h \,\mapsto\,\tfrac{1}{2}\,\ii\,\Delta_h u_h\,, \\
&\Bth\colon\,\nVh \to \nVh\colon~~ w_h \,\mapsto\, -\,\ii\,\DPh(|w_h|^2)\,, \\
&\Bh\colon\,\nVh \to \nVh\colon~~ u_h \,\mapsto\, -\,\ii\,\DPh(|u_h|^2)\,u_h\,.
\end{split}
\end{equation}
In this notation,
\begin{itemize}
\item $\EAh(t)\, u_h$ is associated with subproblem~\eqref{eq:EAh-weak-problem},
\item $\EBh(t,u_h)$ is associated with subproblem~\eqref{eq:EBh-weak-problem}.
\end{itemize}
For representing the solution $\psi_h $ of a system of the type
\begin{equation}
\label{eq:ebh-weak-problem}
\begin{cases}
~\scp{L^2}{\partial_t \psi_h}{\phi_h} = \scp{L^2}{\Bth(w_h)\, \psi_h}{\phi_h}
\quad \text{for all} ~~ \phi_h \in \nVh\,,  \\
~\psi_h\big|_{t=0} = u_h\,,
\end{cases}
\end{equation}
we will also employ an analogous notation as for problem~\eqref{eq:eb-problem-linear},
\begin{subequations}
\begin{equation}
\psi_h= \ebh(t,w_h)\,u_h \,.
\end{equation}
Then, analogously as in~\eqref{eq:EB},
\begin{equation}
\label{eq:def-ebh}
\EBh(t,u_h)= \ebh(t,u_h)\,u_h  \,.
\end{equation}
\end{subequations}
For the resulting fully discrete Strang splitting solution we again write
\begin{equation*}
\psi_n \approx \psi(n \t)\,, \quad n=0,1,2,\ldots,
\end{equation*}
determined by the recurrence
\begin{subequations}
\label{eq:Strang-splitting-h}
\begin{equation}
\label{eq:Strang-splitting-h-1}
\psi_n = \nS_h(\t,\psi_{n-1}) = \EAh(\half \t)\,\EBh\big(\t,\EAh(\half \t)\,\psi_{n-1}\big)\,,
\end{equation}
and we again employ a formal notation for the $n$-fold composition,
\begin{equation}
\label{eq:Strang-splitting-h-2}
\psi_n = \nnSh{n} \ini := \underbrace{\nS_h(\t,\,\cdot\,) \circ \cdots \circ \nS_h(\t,\ini)}_{n \text{ times}}\,.
\end{equation}
\end{subequations}

{\revblue
\subsection{Main results} \label{sec:Outline}
{The central interest of this paper is to establish a convergence result for the
splitting finite element discretization of the Schr{\"o}dinger--Poisson equation
\eqref{eq:SPE}. Here we give a brief overview of the structure
of our convergence proof and state the resulting theorem.
The detailed convergence analysis is worked out in Sec.~\ref{sec:conv-ana}.}

In order to study the global error $\psi_n- \psi(t_n)$
we separate the terms associated with space and time discretization, respectively.
With $\psi_n=\nnSh{n} \Ih\,\ini$ and $\psi(t_n)=\EF(t_n,\ini)$, we write
\begin{equation}
\label{eq:global-parts-overview}
\psi_n - \psi(t_n) = \big( \nnSh{n}\Ih\,\ini - \nnS{n}\ini \big)
                     + \big( \nnS{n}\ini - \EF(t_n,\ini) \big)\,.
\end{equation}
The first term represents the error attributable to the space discretization and the second term is the splitting error
at the semi-discrete level.
\begin{itemize}
\item
The first term in~\eqref{eq:global-parts-overview} is expanded into a telescoping sum
in the following way:
\begin{subequations}
\begin{align}
\label{eq:global-space-overview}
\qquad\qquad
\nnSh{n} \Ih\,\ini - \nnS{n} \ini &=
\nnSh{n} (\Ih-\Ph)\ini
+ (\nnSh{n} \Ph - \Ph\,\nnS{n})\ini
+ (\Ph\,-\mathrm{Id})\,\nnS{n}\ini \notag \\
&= \nnSh{n} (\Ih-\Ph)\ini
+ \sum_{j=1}^{n} \nnSh{n-j}(\nnSh{}\Ph  - \Ph\,\nS)\,\nnS{j-1} \ini
+ (\Ph\,-\mathrm{Id})\,\nnS{n}\ini\,.
\end{align}
We combine a stability argument for the fully discrete splitting operator $\nS_h$
(see Sec.~\ref{sec:S-stab}) with the approximation properties of the finite-element interpolants
(see Theorem~\ref{thm:Brenner-Scott1})
and the Rayleigh-Ritz-projection $\Ph$ (see Theorem~\ref{thm:Brenner-Scott-Ph-conv}).
What remains to be estimated are terms of the form $(\nnSh{}\Ph - \Ph\,\nS) u$,
which is worked out in Sec.~\ref{sec:interpS-err}
(see Theorem~\ref{thm:semi-discrete-error-space})
\item
The second term in~\eqref{eq:global-parts-overview} can similarly be recast as
\begin{equation}
\label{eq:global-time-overview}
\nnS{n} \ini - \EF(t_n,\ini) =
\sum_{j=1}^{n} \nnS{n-j}(\nnS{} \EF(t_{j-1},\ini) - \EF(\t,\EF(t_{j-1},\ini))\,.
\end{equation}
\end{subequations}
Here apply a standard argument for estimating the splitting error
at the semi-discrete level combining the stability of the
splitting operator $\nS$ (see Sec.~\ref{sec:S-stab})
with an estimate for the local splitting error ${\nS(\t,\psi) - \EF(\t,\psi)}$
(see Theorem~\ref{thm:Lubich-conv} or~\cite{lubich07}).
\end{itemize}
This leads to the following global error bound for the full discretization.
\begin{theorem}
	\label{thm:convergence-complete}
	Suppose that $\psi\in H^\ell$, $\ell\geq 4$ and that $\Omega$
    is such that~\eqref{eq:Omega-property} holds.
    Consider the fully discretized method from~\eqref{eq:Strang-splitting-h-1}
    based on the Strang splitting scheme and conforming finite elements of degree $p$, then
	\begin{align}
		\label{eq:Global-error-FEM}
		\begin{aligned}
			\|\psi_n - \psi(t_n)\|_{L^2} & \leq C\, t_n\big(\t^2 + h^{s}\,(1+\tfrac{1}{\t})^\beta\big)\,, \\
            \|\psi_n - \psi(t_n)\|_{H^1} & \leq C\, t_n\big(\t + h^{s-1}\,(1+\tfrac{1}{\t})^\beta\big)\,,
		\end{aligned}
	\end{align}
	where $s=\min\{\ell,p+1\}$ and $\beta=\max\{0,\sgn(p+3-\ell)\}$.
    Here, $C$ depends on $\Omega$, $d$, the $H^4$- and the $H^{s+2(1-\beta)}$-norms of $\psi$. 
\end{theorem}
The proof of Theorem~\ref{thm:convergence-complete} is based on the combination of Theorem~\ref{thm:semi-discrete-error-space}
for the contribution $h^{s}(1+\tfrac{1}{\t})^\beta$ and on Theorem~\ref{thm:Lubich-conv} for the contribution of $\t^2$.

\paragraph{Conclusions} From Theorem~\ref{thm:convergence-complete}, we can deduce the following convergence properties:
\begin{itemize}
	\item For an initial value $\ini\in H^{p+3}$, we obtain the classical convergence order in $\t$ and $h$,
	\begin{equation*}
        \|\psi_n - \psi(t_n)\|_{L^2} \leq C\, t_n\big(\t^2 + h^{p+1}\big)\,, \qquad
		\|\psi_n - \psi(t_n)\|_{H^1} \leq C\, t_n\big(\t + h^{p}\big)\,.
	\end{equation*}
	\item
    {\revred For an initial value $\ini\in H^\ell$, $\ell<p+3$
    we obtain convergence of order $\nO(h^{s})$
    respectively $\nO(h^{s-1})$ in space, but
    with a possibly reduced convergence order in time
    (depending in the ratio between $ \tau $ and $ h $),
	\begin{equation*}
        \|\psi_n - \psi(t_n)\|_{L^2} \leq C\, t_n\big(\t^2 + \tfrac{h^{s}}{\t} \big)\,, \qquad
		\|\psi_n - \psi(t_n)\|_{H^1} \leq C\, t_n\big(\t + \tfrac{h^{s-1}}{\t} \big)\,,
	\end{equation*}
	where $s=\min\{\ell,p+1\}$.}
\end{itemize}
}


\section{Convergence analysis}
\label{sec:conv-ana}

\subsection{Global error bound}
\label{sec:full-semi-split}
We start by separating the effects of space and time discretization,
see~\eqref{eq:global-space-overview} and~\eqref{eq:global-time-overview},
and consider bounds in the $H^1$- and $L^2$-norm.

By a Lady Windermere's fan argument and the stability estimates
from Sec.~\ref{sec:S-stab}, the expression
$\| \nnSh{n}\Ih\,\ini - \nnS{n}\ini\|$ in~\eqref{eq:global-space-overview}
can be expressed by an $h^s$-bound in $L^2$ and an $h^{s-1}$-bound in $H^1$, as shown in the following Theorem~\ref{thm:semi-discrete-error-space}.
The norms $\| \nnS{n}\ini - \EF(t_n,\ini)\|$ have already been studied in~\cite{lubich07}
and are summarized in Theorem~\ref{thm:Lubich-conv} below.
{\revred
This implies the main convergence result stated in
Theorem~\ref{thm:convergence-complete},
where error bounds depending on the regularity
of the initial values are given.}

In our convergence theory we make use of several stability estimates and consistency results which are collected
in Sec.~\ref{sec:S-stab}--\ref{sec:regularity-full} below.
{\revred Several auxiliary results and estimates
are collected in the appendix.}

\begin{theorem}
	\label{thm:semi-discrete-error-space}
	Let $\ini\in H^{\ell}$ for $\ell\geq 4$, $\max\limits_{1 \leq m \leq n}\|\nnS{m}\ini\|_{H^s}\leq M_s$,
    and $\max\limits_{1 \leq m \leq n}(\|\nnSh{m}\Ph\,\nnS{n-m} u\|_{H^1})\leq a_{\nS_h}$.
	Then for $s=\min\{p+1,\ell\}$, the $L^2$ and $H^1$ bounds of the semi-discrete error
	$ \nnSh{n} \Ih\,\ini - \nnS{n} \ini$ can be bounded in $L^2$ and $H^1$:
	\begin{align*}
		\| \nnSh{n}\Ih\,\ini - \nnS{n}\ini\|_{L^2}
		&\leq C\, t_n\, h^s\,(1+\tfrac{1}{\t})^\beta\,, \\
		\| \nnSh{n} \Ih\,\ini - \nnS{n}\ini\|_{H^1}
		&\leq  C\, t_n\, h^{s-1}\,(1+\tfrac{1}{\t})^\beta\,,
	\end{align*}
	where $\beta=\max\{0,\sgn(p+3-\ell)\}$ and $C$ depends on $d$, $\Omega$, $t_n$, $a_{\nS_h}$ and $M_{s+2(1-\beta)}$.
\end{theorem}

\begin{proof} $~$ \\[-2mm]
	\begin{itemize}
		\item {\em $ L^2 $-bound.}
		We proceed as indicated at the beginning of Sec.~\ref{sec:Outline}
        and use the stability properties~\eqref{eq:Sh-stability}
        of the splitting operator $\nS_h$ (see Proposition~\ref{thm:Sh-stability} in Sec.~\ref{sec:S-stab}):
		\begin{subequations}
			\begin{align}
				&\| \nnSh{n} \Ih\,\ini - \nnS{n} \ini \|_{L^2}
				\leq \| \nnSh{n}(\mathrm{Id} - \Ih)\ini \|_{L^2} + \| \nnSh{n}(\mathrm{Id} - \Ph) \ini \|_{L^2} \notag \\
				& \qquad {} + \sum_{j=1}^{n} \| \nnSh{n-j} \nnSh{} \Ph \nnS{j-1} \ini- \nnSh{n-j} \Ph\,\nS\,\nnS{j-1} \ini \|_{L^2}
				+  \| (\mathrm{Id}-\Ph)\,\nnS{n}\ini \|_{L^2} \notag \\
				& \quad \leq \ee^{C\,t_n\, a_{\nS_h}^2}
				\big(\|(\mathrm{Id} - \Ih)\ini\|_{L^2} + \|(\mathrm{Id} - \Ph)\ini\|_{L^2}\big) \notag \\
				& \qquad {} + \sum_{j=1}^{n} \ee^{C\,(t_n-t_j)\,a_{\nS_h}^2}
				\| \nnSh{} \Ph\,\nnS{j-1} \ini - \Ph\,\nS\,\nnS{j-1} \ini \|_{L^2}
				+ \| (\mathrm{Id}-\Ph)\,\nnS{n}\ini \|_{L^2} \notag \\
				& \quad \leq \ee^{C\,t_n\, a_{\nS_h}^2}
				\big(\|(\mathrm{Id} - \Ih)\ini\|_{L^2} + \|(\mathrm{Id} - \Ph)\ini\|_{L^2}\big) \label{eq:semi-FEM-error-line1}\\
				& \qquad {} + n\,\ee^{C\,t_n\,a_{\nS_h}^2}\max_{1 \leq j\leq n}
				\|\nnSh{} \Ph\,\nnS{j-1} \ini - \Ph\,\nS\,\nnS{j-1} \ini \|_{L^2}
				+ \| (\mathrm{Id}-\Ph)\,\nnS{n}\ini \|_{L^2}\,.
			\end{align}
		\end{subequations}
		By the regularity result for the splitting operator $\nS_h$,
        see Lemma~\ref{thm:Sh-regularity} in Sec.~\ref{sec:regularity-full},
        we can ensure the existence of the constant $a_{\nS_h} $.
			
		The expressions in~\eqref{eq:semi-FEM-error-line1} can be bounded using
		Theorems~\ref{thm:Brenner-Scott1} and~\ref{thm:Brenner-Scott-Ph-conv},
		\begin{align*}
			\|(\mathrm{Id} - \Ph)\,u\|_{L^2} & \leq C\,h^{s}\,\|u\|_{H^s}\,, \\
			\|(\mathrm{Id} - \Ih)\,u\|_{L^2} & \leq C\,h^{s}\,\|u\|_{H^s}\,,
		\end{align*}
		and by the bound~\eqref{eq:Hm-reg-bound} from Proposition~\ref{thm:Hm-regularity} we obtain
		\begin{equation*}
			\|\nnS{n}\ini \|_{H^s} \leq \ee^{ L_s\, t_n}\,\| \ini\|_{H^s}\,.
		\end{equation*}
		It remains to bound $\| \nnSh{} \Ph\,\nnS{j-1} \ini - \Ph\,\nS\,\nnS{j-1} \ini \|_{L^2}$.
        By Theorem~\ref{thm:interpS-err} from Sec.~\ref{sec:interpS-err}, it follows that
		\begin{equation*}
			\| \nS_h(\t,\Ph\,u) - \Ph\,\nS(\t,u) \|_{L^2} \leq C\,\t\,(1+\tfrac{1}{\t})^\beta\, h^s\,,
		\end{equation*}
		for $s=\min\{l,p+1\}\,,\,\beta=\max\{0,\sgn(p+3-l)\}$, where $C$ depends on $d$, $\Omega$, $L_s$, and $M_{s+2(1-\beta)}$. Altogether, we obtain
		\begin{align*}
			\| \nnSh{n} \Ih\,\ini - \nnS{n} \ini \|_{L^2} & \leq  \ee^{C\,t_n\, a_{\nS_h}^2}\,h^s\,\| \ini\|_{H^s}   + \ee^{t_n\, L_s}\,h^s\,\|  \ini \|_{H^s} + n\,\ee^{C\,t_n\,a_{\nS_h}^2}\,C\,\t\,(1+\tfrac{1}{\t})^\beta\, h^s\\
			& \leq C\, h^{s}\, \big( \ee^{C\,t_n\, a_{\nS_h}^2} \, \| \ini\|_{H^s} + t_n\,(1+\tfrac{1}{\t})^\beta \, \ee^{C\,t_n\,a_{\nS_h}^2}\big) \\
			& \leq C\, t_n\,(1+\tfrac{1}{\t})^\beta\, h^s\,,
		\end{align*}
		which concludes the proof for the $L^2$ bound.
				
		\item {\em $ H^1 $-bound.} Analogously as for the $L^2$-bound, we use Theorems~\ref{thm:Brenner-Scott1} and~\ref{thm:Brenner-Scott-Ph-conv} and Proposition~\ref{thm:Hm-regularity} and Theorem~\ref{thm:interpS-err}. Hence we obtain
		\begin{subequations}
			\begin{align*}
				\qquad\quad	\| \nnSh{n} \Ih\,\ini - \nnS{n} \ini\|_{H^1}
				& \leq \|\nnSh{n} \Ih \ini- \nnSh{n}\ini\|_{H^1} + \| \nnSh{n}\ini - \nnSh{n} \Ph \ini\|_{H^1}\notag \\
				& \quad {} + \sum_{j=1}^{n} \big\| \nnSh{n-j} \nnSh{} \Ph\,\nnS{j-1}\ini -\nnSh{n-j} \Ph\, \nnS{}\,\nnS{j-1}\ini \big\|_{H^1}  + \| \Ph\,\nnS{n}\ini - \nnS{n} \ini \|_{H^1}\notag\\
				& {} \leq \ee^{C\,t_n\, a_{\nS_h}^2} \big(\|\Ih\,\ini - \ini\|_{H^1} + \|\ini - \Ph\,\ini\|_{H^1}\big)
                  + \| \Ph\,\nnS{n} \ini - \nnS{n} \ini \|_{H^1}\notag\\
				& \quad {} + n\,\ee^{C\,t_n\,a_{\nS_h}^2}\max_{1 \leq j\leq n}
				\big\|  \nnSh{} \Ph\,\nnS{j-1} \ini- \Ph\,\nnS{} \nnS{j-1} \ini \big\|_{H^1} \notag\\
				& {} \leq \ee^{C\,t_n\, a_{\nS_h}^2} h^{s-1}\,\| \ini\|_{H^s}   + \ee^{t_n\, L_s}h^{s-1}\,\|  \ini \|_{H^s} + n\,\ee^{C\,t_n\,a_{\nS_h}^2}\,C\,\t\,(1+\tfrac{1}{\t})^\beta\, h^{s-1}\\
				& {}\leq  C\, h^{s-1}\, \big( \ee^{C\,t_n\, a_{\nS_h}^2} \, \| \ini\|_{H^s} + t_n\,(1+\tfrac{1}{\t})^\beta \, \ee^{C\,t_n\,a_{\nS_h}^2}\big) \\
				& {}\leq C\, t_n\,(1+\tfrac{1}{\t})^\beta\, h^{s-1}\,,
			\end{align*}
		\end{subequations}
		where the constant $C$ depends on $a_{\nS_h}$, $t_n$, and $\| \ini\|_{H^s}$.
	\end{itemize}
\end{proof}

The following theorem summarizes the semidiscrete error in time:
\begin{theorem}
	\label{thm:Lubich-conv}
	Suppose that the exact solution $\psi(t_n)$ to the Schr{\"o}dinger-Poisson
	equation~\eqref{eq:SPEproblem} is in $H^2$ for $0\leq t_n \leq T$.
	Then, the semi-discrete numerical solution $\nnS{n}\ini$ given
	by the Strang splitting scheme~\eqref{eq:Strang-splitting} with
	stepsize $\t$  satisfies
	\begin{subequations}
		\label{eq:Lub-proof12}
		\begin{align}
			\| \nnS{n}\ini - \psi(t_n)\|_{L^2} & \leq C_1\,\t^2\,,\label{eq:Lub-proof2} \\
			\| \nnS{n}\ini - \psi(t_n)\|_{H^1} & \leq C_2\,\t\,, \label{eq:Lub-proof1}
		\end{align}
	\end{subequations}
	where both constants $C_1$, $C_2$ depend on $T$, $\Omega$, and on the $H^2$-norm of $\psi$.
\end{theorem}

\begin{proof}
	The detailed proof can be found in~\cite{lubich07} with the restriction that $C_1$ and $C_2$
	depend on the $H^3$- respectively the $H^4$-norm of $\psi$.
	For the improved bounds~\eqref{eq:Lub-proof12} in $H^2(\RR^d)$ we refer to~\cite{koclub08c} and for full details on the computation of the commutators see~\cite{koclub08a}.
	In our case, the domain $\Omega$ is finite, but the analogous Sobolev embeddings hold also in this case, see~\cite{adams75}.
	
	For the dominant terms
	we now show the sharp estimates directly. For the bound~\eqref{eq:Lub-proof1},
    the dependence on the $H^3$-norm is indicated in~\cite{lubich07} to arise from a bound of
	\begin{equation*}
		\|\DP(\psi\,\ol{\Delta \psi})\psi\|_{H^1}\,.
	\end{equation*}
	Since all other terms are already bounded in terms of $ \| \psi \|_{H^2} $, it suffices to estimate this term likewise.
	Using Proposition~\ref{thm:Hmin1-est} we obtain
	\begin{align*}
		\|\DP(\psi\,\ol{\Delta \psi})\,\psi\|_{H^1}&\leq \| \nabla\DP(\psi \,\ol{\Delta \psi})\,\psi\|_{L^2}+\|\DP(\psi\,\ol{\Delta \psi})\,\nabla\psi\|_{L^2}\\
		&\leq \| \DP(\psi\,\ol{\Delta \psi})\|_{H^1}\,\|\psi\|_{L^\infty} + C\,\|\DP(\psi\,\ol{\Delta \psi})\|_{H^1}\,\|\nabla\psi\|_{H^1}\\
		& \leq C\,\| \psi\,\ol{\Delta \psi}\,\|_{H^{-1}}\,\|\psi\|_{H^2}  \leq C\,\| \psi \|_{H^1}\,\| \ol{\Delta \psi}\,\|_{L^2}\,\|\psi\|_{H^2} \\
		& \leq C\,\|\psi\|_{H^2}^2\,\|\psi\|_{H^1}\,.
	\end{align*}
	For the bound~\eqref{eq:Lub-proof2} in terms of $ \| \psi \|_{H^2} $ we refer to the bounds given in~\cite{lubich07}
	and the improved bounds from~\cite{koclub08c}.
	For full details on the computation of the commutators involved, see~\cite{koclub08a}.
	
	However, the critical term in the commutator bound is identified as
	\begin{equation*}
		\|4\,\ii\,\DP(\psi \,\Delta^2 \ol{\psi})\,\psi\|_{L^2}\,,
	\end{equation*}
	for which we will show in detail that it can be bounded in terms of the $H^2$-norm.
	Using a duality argument in $L^2$ and integration by parts we obtain
	\begin{align*}
		\|\DP( \psi\,\Delta^2\ol{\psi})\,\psi\|_{L^2}
		& \leq C\,\| \DP(\psi\,\Delta^2 \ol{\psi})\|_{L^2}\,\| \psi\|_{H^2}
          \leq C\,\sup_{\|\phi\|_{L^2}=1} (\DP(\psi\,\Delta^2 \ol{\psi}),\phi)_{L^2}\,\| \psi\|_{H^2} \\
		& = \sup_{\|\phi\|_{L^2}=1} (\psi\,\Delta^2 \ol{\psi},\DP\phi)_{L^2}\,\|\psi\|_{H^2}  \leq C\,\sup_{\|\phi\|_{L^2}=1} (\Delta^2 \ol{\psi},\ol{\psi}\,\DP\phi)_{L^2}\cdot\|\psi\|_{H^2} \\
		&  = C\,\sup_{\|\phi\|_{L^2}=1} (\Delta \ol{\psi},\Delta(\ol{\psi}\,\DP\phi))_{L^2}\,\cdot\,\|\psi\|_{H^2} \\
		& \leq C\,\sup_{\|\phi\|_{L^2}=1} \int_{\Omega} (\Delta {\ol\psi}) \big( (\Delta \psi)\,\DP(\ol{\phi})
          + (\nabla\psi)\cdot (\nabla\DP(\ol{\phi}))
		+ \psi\,\ol{\phi} \big)\,\dd x\,\cdot\,\| \psi\|_{H^2} \\
		& \leq C\,\sup_{\|\phi\|_{L^2}=1}\big( \|\Delta \psi\|_{L^2}\,\|\Delta \psi\|_{L^2}\,\| \DP(\phi)\|_{L^\infty}
		+ \| \Delta \psi\,\|_{L^2}\,\| \nabla\psi\|_{L^4}\,\| \nabla\DP(\phi)\|_{L^4}\\[-2mm]
		& \qquad\qquad \qquad {} +\| \Delta \psi\,\|_{L^2}\,\| \psi\|_{L^\infty}\,\| \phi\|_{L^2}\big)\,\cdot\,\|\psi\|_{H^2} \\
		& \leq C\,\sup_{\|\phi\|_{L^2}=1}\big( \| \psi\|_{H^2}\,\|\psi\|_{H^2}\,\|\DP(\phi)\|_{H^2}
		+ \|\psi \|_{H^2}\,\| \nabla\psi\|_{H^1}\,\| \DP(\phi)\|_{H^2} \\[-2mm]
		& \qquad \qquad \qquad {} + \| \psi\|_{H^2}\,\| \psi\|_{H^2}\,\|\phi\|_{L^2}\big)\,\cdot\,\|\psi\|_{H^2}  \\
		& \leq C\,\| \psi\|_{H^2}^3\,,
	\end{align*}
	concluding the proof.
\end{proof}

\subsection{Stability properties of the splitting  operators}
\label{sec:S-stab}

To reduce the analysis of the global error to the study of the splitting error in a single time step,
the following stability estimates for the splitting operators $\nS$ and $\nS_h$ are required.
Since, analogously as in~\cite{lubich07},
the $L^2$- and $H^1$-bounds depend on the $H^1$-norms of the numerical solution,
we first need a stability and convergence result in $H^1$
to show the $H^1$-boundedness of the numerical solution.
We list $L^2$- and $H^1$-bounds together and verify the $L^2$-bounds in hindsight.

\begin{proposition}
	\label{thm:Sh-stability}
    The fully discretized splitting operator $\nS_h$ defined in~\eqref{eq:Strang-splitting-h-1}
    enjoys $H^1$-stability and $H^1$-conditional $L^2$-stability,
    \begin{subequations}
    \label{eq:Sh-stability}
	\begin{align}
	\label{eq:Sh-stability-1}
	\big\|\nS_h (\t,\wu) - \nS_h(\t,u)\big\|_{L^2} & \leq \ee^{C\,\t\,a_{\nS_h}^2 } \|\wu-u\|_{L^2}\,, \\
    \label{eq:Sh-stability-2}
    \big\|\nS_h (\t,\wu) - \nS_h(\t,u)\big\|_{H^1} & \leq \ee^{C\,\t\,a_{\nS_h}^2 } \|\wu-u\|_{H^1}\,,
	\end{align}
    \end{subequations}
	for $u,\wu\in \nV^h$, with $a_{\nS_h} = \max\{\| u\|_{H^1},\|\wu\|_{H^1}\}$,
    and $C$ depending on $h$, $d$, and $\Omega$.
\end{proposition}

\begin{proposition}
	\label{thm:S-stability}
	The semi-discrete splitting operator $\nS$ defined in~\eqref{eq:Strang-splitting-1}
    enjoys $H^1$-stability and $H^1$-conditional $L^2$-stability,

	\begin{align*}
	\| \nS (\t,\wu) - \nS(\t,u)\|_{L^2}  & \leq \ee^{C\,\t\,a_{\nS}^2 }\,\|\wu-u \|_{L^2}\,, \\
    \| \nS (\t,\wu) - \nS(\t,u)\|_{H^1}  & \leq \ee^{C\,\t\,a_{\nS}^2 }\,\|\wu-u \|_{H^1}\,,
	\end{align*}
	for $u,\wu\in H^1_{0}$, with $a_{\nS}=\max\{\| u\|_{H^1},\|\wu\|_{H^1}\}$,
    and $C$ depending on $d$ and $\Omega$.
\end{proposition}

\begin{proof}[Proof of $L^2$-stability in Proposition~\ref{thm:Sh-stability}]
	{\revred Our goal is to find} an estimate of $\nS_h(\t,\wu)-\nS_h(\t,u)$
    for two functions $u$, $\wu \in \nV^h$.
\begin{itemize}
\item
    We combine the unitarity of the operators $\EAh$ and $\EBh$
    (see Proposition~\ref{thm:discrete-conservation},
    \eqref{eq:discrete-conservation-A-L2} and~\eqref{eq:discrete-conservation-B-L2})
    with the linearity of $\EAh$ and the definition of
    $\EBh(\t,u)=\ebh(\t,u)u$ with the linear operator $ \ebh(\t,\,\cdot\,) $ from~\eqref{eq:def-ebh}.
    With the abbreviations $ w_h = \eAh u $, $ \wt{w}_h = \eAh\wu $ we obtain
	\begin{subequations}
		\begin{align*}
		\| \nS_h (\t,\wu) - \nS_h(\t,u)\|_{L^2}
        & = \|\eAh \big( \EBh(\t,\wt{w}_h ) - \EBh(\t,w_h) \big)\|_{L^2} \\
		&  
		   = \| \ebh(\t,\wt{w}_h )\wt{w}_h - \ebh(\t,w_h)\,w_h \|_{L^2} \\
		& \leq \| \ebh(\t,\wt{w}_h )(\wt{w}_h -w_h)\|_{L^2}
		    + \|\ebh(\t,\wt{w}_h )\,w_h - \ebh(\t,w_h)\,w_h \|_{L^2} \\
        & =\|\wt{w}_h -  w_h\|_{L^2} +  \| \ebh(\t,\wt{w}_h )\,w_h - \EBh(\t,w_h)\|_{L^2} \\
		& =\|\wu - u\|_{L^2} +  \| \ebh(\t,\wt{w}_h )\,w_h - \EBh(\t,w_h)\|_{L^2} \,.
		\end{align*}
	\end{subequations}
	To estimate the last term we use the
    {\revred mild formulation}~\eqref{eq:B-B-diff},
	\begin{align}
	&\| \ebht{\wt{w}_h} w_h - \EBh(\t, w_h) \|_{L^2}\notag \\
	&\qquad \leq \int_0^\t \big\|\ebh(\t-\s,\wt{w}_h)\big( \Bt_h(\wt{w}_h) - \Bt_h(w_h)\big)\,\ebh(\s,w_h)\,w_h \big\|_{L^2}\,\dd \s \notag \\
	&\qquad \leq \t\,\sup_{0\leq \s \leq\t} \big\|\big(\Bt_h(\wt{w}_h) - \Bt_h(w_h)\big)\,\ebh(\s,w_h)\,w_h \big\|_{L^2}\label{eq:Err-prop-est}\,.
	\end{align}	
\item
	To find an estimate for the right-hand side in~\eqref{eq:Err-prop-est}
    we denote $z_h=\Bt_h(w_h )$,\, $\wt{z}_h=\Bt_h(\wt{w}_h )$, and consider the Galerkin equations
		\begin{align*}
		(\nabla z_h,\nabla v_h)_{L^2} & = -(|w_h|^2,v_h)_{L^2}
        \quad \text{for all}~~ v_h \in \nV^h\,, 
     \\
		(\nabla \wt{z}_h,\nabla v_h)_{L^2} & = -(|\wt{w}_h|^2,v_h)_{L^2}
        \quad \text{for all}~~ v_h \in \nV^h\,. 
		\end{align*}
	The difference
	\begin{equation*}
	(\nabla (\wt{z}_h-z_h),\nabla v_h)_{L^2} = -(|\wt{w}_h|^2-|w_h|^2,v_h)_{L^2}
    \quad \text{for all}~~ v_h \in \nV^h\,,
	\end{equation*}	
	is again a discrete Poisson problem with solution
	\begin{equation*}
	\wt{z}_h - z_h = \DPh(|\wt{w}_h|^2- |w_h|^2)\,.
	\end{equation*}
	With this observation 
    we obtain for~\eqref{eq:Err-prop-est}:
	\begin{align}
	&\t\,\sup_{0\leq \s \leq\t}\|(\Bt_h(\wt{w}_h) - \Bt_h(w_h)) |\ebh(\s,w_h)\,w_h| \|_{L^2}  \notag \\
	& \qquad \qquad  = \t\,\sup_{0\leq \s \leq\t} \| \DPh (|\wt{w}_h|^2 -|w_h|^2 ) |\ebh(\s,w_h)\,w_h|\|_{L^2} \notag\\
	& \qquad \qquad \leq \t \,\sup_{0\leq \s \leq\t} \| \DPh (|\wt{w}_h-w_h| \cdot | \wt{w}_h+w_h| )|\ebh(\s,w_h)\,w_h|\|_{L^2}\,.	\label{eq:Bt-linear}
	\end{align}
	Proposition~\ref{thm:DPh-L2-est} in the appendix yields
	\begin{align*}
	&\| \DPh (|\wt{w}_h-w_h| \cdot | \wt{w}_h+w_h| )|\ebh(\s,w_h)\,w_h|\|_{L^2} \\
	& \qquad \qquad {} \leq C\,\| \wt{w}_h-w_h\|_{L^2}\,\| \wt{w}_h+w_h \|_{H^1} \| \ebh(\s,w_h)\,w_h\|_{H^1}\,.
	\end{align*}
	Inserting $w_h=\eAh u$ and $\wt{w}_h=\eAh \wu$ again and using Proposition~\ref{thm:EBh-H1-est} yields
	\begin{align*}
	& \| \ebh(\t,\eAh\wu )\,\eAh u - \EBh(\t,\eAh u ) \|_{L^2} \\
	& \qquad \leq \t\, C\,\| \eAh (\wu -u)\|_{L^2}\,\| \eAh (\wu +u)\|_{H^1}\,\ee^{\t\, C\,\| u\|_{H^1}^2}\,\| u\|_{H^1}\\
	& \qquad \leq \t\, C\,\| \wu - u\|_{L^2}\,( \| \wu\|_{H^1} + \| u \|_{H^1})\,\ee^{\t\, C\,\| u\|_{H^1}^2}\,\| u\|_{H^1}\\
	& \qquad \leq C\,\t\, a_{\nS_h}^2\,\ee^{\t\, C\,\| u\|_{H^1}^2}\,\| \wu -u\|_{L^2} \,.
	\end{align*}
	With $1+x\leq \ee^x$ we finally obtain
	\begin{equation*}
	\| \nS_h (\t,\wu) - \nS_h(\t,u)\|_{L^2}
    \leq \ee^{C\,\t\,a_{\nS_h}^2} \|\wu-u\|_{L^2}\,,
	\end{equation*}
	with $a_{\nS_h} = \max\{\| u \|_{H^1}, \|\wu \|_{H^1}\}$.  To show the boundedness of the constant $a_{\nS_h}$ in $H^1$, we further need the $H^1$ stability result (Proposition~\ref{thm:Sh-stability}),
    the $H^1$ interpolation error (Theorem~\ref{thm:interpS-err}) and the regularity of the $H^1$ solution
    (Lemma~\ref{thm:Sh-regularity}).
\end{itemize}
\end{proof}

\begin{proof}[Proof of the $H^1$-stability in Proposition~\ref{thm:Sh-stability}]
	
	We start similarly as for the $L^2$ case:

	We combine the unitarity of the operator $\EAh$ in $H^1$ (see Proposition~\ref{thm:discrete-conservation}, \eqref{eq:discrete-conservation-A-H1})
	with the linearity of $\EAh$ and the definition of  $\EBh(\t,u)=\ebh(\t,u)u$ with the linear operator $ \ebh(\t,\,\cdot\,) $ from~\eqref{eq:def-ebh}.
		With the abbreviations $ w_h = \eAh u $, $ \wt{w}_h = \eAh\wu $ we obtain
	\begin{subequations}
		\begin{align}
		\| S_h (\t,\wu) - S_h(\t,u) \|_{H^1}
		&= \|\eAh \EBh(\t,\wt{w}_h) -\eAh \EBh(\t,w_h)\|_{H^1} \notag \\
        &  = \| \ebh(\t,\wt{w}_h )\wt{w}_h - \ebh(\t,w_h)\,w_h \|_{H^1} \notag \\
		&\leq \|\ebh(\t,\wt{w}_h)\wt{w}_h-\ebh(\t,\wt{w}_h)\,w_h\|_{H^1}
		       + \|\ebh(\t,\wt{w}_h)\,w_h -\ebh(\t,w_h)\,w_h\|_{H^1} \notag\\
		&= \| \ebh(\t,\wt{w}_h)\,\eAh(\wu - u)\|_{H^1} \label{eq:Sh-H1-line-a}\\
		& \quad {} + \|\ebh(\t,\wt{w}_h)\,w_h - \ebh(\t,w_h)\,w_h \|_{H^1}\,.\label{eq:Sh-H1-line-b}
		\end{align}
	\end{subequations}
    Now we separately estimate~\eqref{eq:Sh-H1-line-a} and~\eqref{eq:Sh-H1-line-b}.
	First, from Proposition{~\ref{thm:EBh-H1-est}} we obtain for~\eqref{eq:Sh-H1-line-a}
	\begin{equation*}
	\| \ebh(\t,\wt{w}_h)\,\eAh(\wu - u) \|_{H^1}
    \leq \ee^{\t\,  C \| \wu\|_{H^1}^2} \|\wu - u \|_{H^1}\,.
	\end{equation*}
	To obtain a bound for~\eqref{eq:Sh-H1-line-b}, we use the linear variation-of-constant formula as in~\eqref{eq:B-B-diff} and
    Proposition~\ref{thm:EBh-H1-est},
	\begin{subequations}
		\begin{align}
		& \|\ebh(\t,\wt{w}_h)\,w_h - \ebh(\t,w_h)\,w_h \|_{H^1} \notag \\
		& \qquad \leq \t \sup_{0\leq \s\leq \t} \|\ebh(\t-\s,\wt{w}_h)(\Bt_h(\wt{w}_h ) - \Bt_h(w_h))
		                                        \cdot \ebh(\s,w_h)\,w_h\|_{H^1} \notag\\
		& \qquad \leq \t \sup_{0\leq \s\leq \t} \ee^{\t\, C \|\wt{w}_h\|_{H^1}^2}\|(\Bt_h(\wt{w}_h ) - \Bt_h(w_h))\,\ebh(\s,w_h)\,w_h\|_{H^1} \notag\\
		& \qquad \leq \t \sup_{0\leq \s\leq \t} \ee^{ \t\, C \|\wu\|_{H^1}^2}
          \big(\|(\nabla(\Bt_h(\wt{w}_h ) - \Bt_h(w_h)))\,\ebh(\s,w_h)\,w_h\|_{L^2} \label{eq:Sh-H1-line-c} \\[-2mm]
		& \qquad\qquad\qquad\qquad\qquad\qquad {} + \|(\Bt_h(\wt{w}_h ) - \Bt_h(w_h))\nabla(\ebh(\s,w_h)\,w_h )\|_{L^2} \big)\label{eq:Sh-H1-line-d}\,.
		\end{align}
	\end{subequations}
	For~\eqref{eq:Sh-H1-line-c} we use the bound~\eqref{eq:DPh-H1-est-corr} from Corollary~\ref{thm:DPh-H1-est-corollary},
	\begin{align*}
	\|(\nabla(\Bt_h(\wt{w}_h) - \Bt_h(w_h)))\,\ebh(\s,w_h)\,w_h\|_{L^2}
	&= \|(\nabla\DPh(|\wt{w}_h|^2 - |w_h|^2))\,\ebh(\s,w_h)\,w_h\|_{L^2} \\
	&\leq \|(\nabla\DPh(|\wt{w}_h -w_h|\,|\wt{w}_h + w_h|))\,\ebh(\s,w_h)\,w_h\|_{L^2} \\
	&\leq C\,(1+h)\,\| \wt{w}_h -w_h\|_{H^1}\,\| \wt{w}_h + w_h\|_{H^1}\,\|\ebh(\s,w_h)\,w_h\|_{H^1} \\
	&\leq C\,(1+h)\,\|\wu-u\|_{H^1}\,\| \wu+u\|_{H^1}\,\ee^{\t\, C\|u\|_{H^1}^2}\,\| u \|_{H^1}\,.
	\end{align*}
	For~\eqref{eq:Sh-H1-line-d} we use in addition~\eqref{eq:Brenner-Scott-infty}, the $L^\infty$ bound from Theorem~\ref{thm:Brenner-Scott1},
	\begin{align*}
	&\|(\Bt_h(\wt{w}_h) - \Bt_h(w_h))\nabla(\ebh(\s,w_h)\,w_h )\|_{L^2} \\
	& \qquad \leq  \|  \DPh(|\wt{w}_h -  w_h|\, | \wt{w}_h+ w_h|) \|_{L^\infty}\,\| \nabla(\ebh(\s,w_h)\,w_h )\|_{L^2} \\
	& \qquad\leq \big( \|  \DP(|\wt{w}_h -  w_h|\, | \wt{w}_h+ w_h|)  - \DPh(|\wt{w}_h -  w_h|\, | \wt{w}_h+ w_h|) \|_{L^\infty} \\
	& \qquad \qquad +  \|  \DP(|\wt{w}_h -  w_h|\, | \wt{w}_h+ w_h|) \|_{L^\infty}\big)  \| \nabla(\ebh(\s,w_h)\,w_h )\|_{L^2} \\
	& \qquad\leq \big(  C\,h^{2-\frac{3}{2}} \|  \DP(|\wt{w}_h -  w_h|\, | \wt{w}_h+ w_h|) \|_{H^2}+ C\,\|  \DP(|\wt{w}_h -  w_h|\, | \wt{w}_h+ w_h|) \|_{H^2} \big)
             \cdot \ee^{\s\, C\,\| w_h\|_{H^1}^2}\,\| w_h\|_{H^1} \\
	& \qquad\leq C\,(1+h^{1/2})\,\| \wt{w}_h- w_h\|_{H^1}\,\| \wt{w}_h +  w_h\|_{H^1}\,\ee^{\t\, C\,\| w_h\|_{H^1}^2}\,\|  w_h \|_{H^1}\\
	& \qquad\leq \| \wt{w}_h- w_h\|_{H^1} \big(\ee^{\t\, C\,\| w_h\|_{H^1}^2}\,C\,(1+h^{1/2}) (\| \wt{w}_h\|_{H^1} + \| w_h\|_{H^1})\,\|  w_h \|_{H^1} \big)\\
	& \qquad\leq \| \wu - u\|_{H^1} \big(\ee^{\t\, C\,\| u\|_{H^1}^2}\,C\,(1+h^{1/2}) (\| \wu \|_{H^1} + \| u\|_{H^1})\,\| u \|_{H^1} \big)\,.
	\end{align*}
	Summarizing, we conclude
	\begin{equation*}
		\| S_h (\t,\wu) - S_h(\t,u)\|_{H^1}
		\leq \| \wt u - u\|_{H^1}\,\ee^{\t\,C\, a_{\nS_h}^2} ( 1+C\,  \t\,a_{\nS_h}^2\,\ee^{\t\, C\,a_{\nS_h}^2} )\,,
		\end{equation*}
		where $a_{\nS_h}=\max\{\| u \|_{H^1}, \| \wu \|_{H^1} \} $ and $C$ depends on $h$, on $d$, and on $\Omega$.
        The dependence on $h$ for $h<1$ is negligible and will be omitted in the further analysis.
		
		With $1+x\leq \ee^x$ we finally obtain
		\begin{equation*}
		\| S_h (\t,\wu) - S_h(\t,u) \|_{H^1} \leq \ee^{C\,\t\, a_{\nS_h}^2} \| \wt u - u\|_{H^1}\,,
		\end{equation*}
		finishing the proof.
\end{proof}

\begin{proof}[Proof of Proposition~\ref{thm:S-stability}]
	This is conducted analogously to the proof of Proposition~\ref{thm:Sh-stability}, substituting $\eA u$ by $w$ and $\eA \wu$ by $\wt{w}$, with the Sobolev embedding replaced by Hardy's inequality (see the proof in~\cite{lubich07}),
	\begin{align*}
	\| \DP (|\wt{w} - w| \cdot | \wt{w}+ w| )\EB(\t,w)\|_{L^2} & \leq C\,\|\wt{w} - w\|_{L^2} \| \wt{w} + w\|_{H^1}  \|\EB(\t,w)\|_{L^2}\\
	& \leq C\,\| \wu - u\|_{L^2}\,\|\wu + u\|_{H^1}\,\|u\|_{L^2}\\
	& \leq C\,\| \wu - u\|_{L^2}\, a_{\nS}^2\,.
	\end{align*}
	
	For the proof of $H^1$-stability,
    which is less involved for $\nS$ than for $\nS_h$, we refer to~\cite{lubich07}.
\end{proof}

\subsection{Consistency of the fully discretized splitting operator}
\label{sec:interpS-err}

The aim of this section is to provide an estimate for $\nS_h(\t,\Ph\,u) -\Ph S(\t,u)$
in the $H^1$- and the $L^2$-norm, which will be presented in Theorem~\ref{thm:interpS-err}.
As an essential part of the proof, we first bound the approximation error of $\EAh$ as compared to $\Ph \EA$ in these norms.

\paragraph{$L^2$- and $H^1$-estimates for $ (\Ph \EA(\t) - \EAh(\t) \Ph)u $}
Propositions~\ref{thm:EAdiff} and~\ref{thm:EAdiff-weak} below specify bounds for
\begin{equation}
\label{eq:e(tau)}
e_A(\t) := (\Ph \EA(\t) - \EAh(\t) \Ph)u \,\in \nV^h\,.
\end{equation}
As before, $ p $ denotes the polynomial degree associated with the FEM subspace $\nV^h$.

Proposition~\ref{thm:EAdiff} requires a higher Sobolev regularity of $s+2$, but offers an additional dependence on $\t$, while Proposition~\ref{thm:EAdiff-weak} requires only a Sobolev regularity of $s$. Both results can be recast together as
\begin{align*}
\| e_A(\t) \|_{L^2}&\leq C\,\t\, (\tfrac{1}{\t})^\beta h^s \|u  \|_{H^{s+2(1-\beta)}}\,,\\
\| e_A(\t) \|_{H^1}&\leq C\,\t\, (\tfrac{1}{\t})^\beta h^{s-1} \|u  \|_{H^{s+2(1-\beta)}}\,,
\end{align*}
where $\beta=\max\{0,\sgn(p+3-\ell)\}$ and $s=\min\{p+1,\ell\}$. Here, the dependence on $(\tfrac{1}{\t})^\beta$ indicates a reduced approximation quality for $\ell< p+3$.

\begin{proposition}
	\label{thm:EAdiff}
	Let $u\in H^{\ell}$ for $\ell\geq 1$.  Then, $ e_A(\tau) $ from~\eqref{eq:e(tau)} satisfies
	\begin{align}
	\| e_A(\t) \|_{L^2}  & \leq C\, \t\, h^s \|u  \|_{H^{s+2}}\,, \label{eq:EAdiff-L2}\\
	\| e_A(\t) \|_{H^1}  & \leq C\, \t\, h^{s-1} \| u \|_{H^{s+2}}\,, \label{eq:EAdiff-H1}
	\end{align}
	where $C$ depends on $d$ and $\Omega$ and $s=\min\{p+1,\ell-2\}$.
\end{proposition}
\begin{proposition}
	\label{thm:EAdiff-weak}
	Let $u\in H^{\ell}$ for $\ell\geq1$.  Then, $ e_A(\tau) $ from~\eqref{eq:e(tau)} satisfies
	\begin{align}
	\| e_A(\t) \|_{L^2}  & \leq C\, h^s \,\|u  \|_{H^{s}}\,,\\
	\| e_A(\t) \|_{H^1}  & \leq C\, h^{s-1}\, \| u \|_{H^{s}}\,,
	\end{align}
	where $C$ depends on $d$ and $\Omega$ and $s=\min\{p+1,\ell\}$.
\end{proposition}

The proof is given after the proof of Proposition~\ref{thm:EAdiff}.

\begin{proof}[Proof of Proposition~\ref{thm:EAdiff}]
	We denote
	\begin{equation*}
	y(t)=\EA(t)u\,, \quad y_h(t) = \EAh(t) \Ph\,u\,,
	\quad \text{such that}~~ e_A(t) = \Ph\,y(t) - y_h(t)\,.
	\end{equation*}
	Here, $ e_A(0)=0 $,
	and $y(t)$ and $y_h(t)$ are the solutions of
	\begin{align*}
	(\partial_t\,y(t), v)_{L^2} + \half\,\ii\,(\nabla y(t),\nabla v)_{L^2} &= 0
	\quad \text{for all} ~~ v \in H^1_0\,, 
	\\
	(\partial_t\,y_h(t),v_h)_{L^2} + \half\,\ii\,(\nabla y_h(t),\nabla v_h)_{L^2} &= 0
	\quad \text{for all} ~~ v \in \nV_h\,. 
	\end{align*}
	Setting $v=v_h=e_A(t)$ and subtracting these equations we obtain
	\begin{equation*}
	(\partial_t (y(t)-y_h(t)),e_A(t))_{L^2}
	+ \half\,\ii\,\big( \nabla y(t) - \nabla y_h(t), \nabla e_A(t))_{L^2} = 0\,.
	\end{equation*}
	With $y(t)-y_h(t) = e_A(t) - (\Ph\,y(t)-y(t))$ this takes the form
	\begin{align*}
	(\partial_t\,e_A(t),e_A(t))_{L^2}
    - (\partial_t (\Ph\,y(t)-y(t)),e_A(t))_{L^2}
	+ \half\,\ii\,(\nabla e_A(t),\nabla e_A(t))_{L^2}
    - (\nabla\Ph\,y(t)-\nabla y(t),\nabla e_A(t))_{L^2} &=0\,.
	\end{align*}
	Due to the property~\eqref{eq:Ph-def} of the Rayleigh-Ritz projection we have
	$(\nabla\Ph\,y(t)-\nabla y(t),\nabla e_A(t))_{L^2}=0$. Thus we obtain
	\begin{align*}
	(\partial_t\,e_A(t),e_A(t))_{L^2}
    + \half\,\ii\,(\nabla e_A(t),\nabla e_A(t))_{L^2}
	&= (\Ph\,\partial_t\,y(t) - \partial_t\,y(t), e_A(t))_{L^2}\,, \\
	\intertext{and its complex conjugate}
	(e_A(t),\partial_t\,e_A(t))_{L^2}
    - \half\,\ii\,(\nabla e_A(t), \nabla e_A(t))_{L^2}
	&= (e_A(t),\Ph\,\partial_t\,y(t) - \partial_t\,y(t))_{L^2}\,.
	\end{align*}
	Adding these equations gives
	\begin{align}
	2\,\Re(\partial_t\,e_A(t),e_A(t))_{L^2}
    &= 2\,\Re\,(\Ph\,\partial_t\,y(t) - \partial_t\,y(t),e_A(t))_{L^2} \,.\label{eq:pt_e_e}
	\end{align}
	For the left-hand side of~\eqref{eq:pt_e_e} we find
	\begin{equation*}
	2\,\Re(\partial_t\,e_A(t),e_A(t))_{L^2}
	= \partial_t\| e_A(t) \|^2_{L^2} = 2\,\| e_A(t) \|_{L^2}\,\partial_t\| e_A(t)\|_{L^2}\,,
	\end{equation*}
	and on the right-hand side of~\eqref{eq:pt_e_e} we apply the Cauchy-Schwarz and H{\"o}lder inequalities,	
	\begin{align*}
	2\,\Re\,(\Ph\,\partial_t\,y(t) - \partial_t\,y(t), e_A(t))_{L^2}
    & \leq 2\,\big(|\Ph\,\partial_t\,y(t) - \partial_t\,y(t)|, |e_A(t)|\big)_{L^2}\\
	& \leq 2\,\| \Ph\,\partial_t\,y(t) - \partial_t\,y(t) \|_{L^2}\,\| e_A(t) \|_{L^2}\,.
	\end{align*}
	Hence,~\eqref{eq:pt_e_e} yields the inequality
	\begin{align*}
	2\,\| e_A(t) \|_{L^2}\,\partial_t\| e_A(t)\|_{L^2}
    &\leq 2\,\| \Ph\,\partial_t\,y(t) - \partial_t\,y(t) \|_{L^2}\,\| e_A(t) \|_{L^2}\,.
	\end{align*}
	Dividing by $2\,\|e_A(t)\|_{L^2}$ we obtain
	\begin{equation*}
	\partial_t\| e_A(t)\|_{L^2} \leq \| \Ph\,\partial_t\,y(t) - \partial_t\,y(t) \|_{L^2} \,.
	\end{equation*}
	Now we apply the bound for the projection operator from Theorem~\ref{thm:Brenner-Scott-Ph-conv}, yielding
	\begin{equation*}
	\partial_t \| e_A(t)\|_{L^2}\leq  \| \Ph\,\partial_t\,y(t) - \partial_t\,y(t) \|_{L^2}
	\leq C\,h^s\,\|  \partial_t\,y(t) \|_{H^{s}}\,,
	\quad s=\min\{p+1,\ell\}\,,
	\end{equation*}
	where $l$ is the degree of maximal Sobolev regularity of $\partial_t\,y(t)$.
	Now we integrate over $\s$ from $0$ to $t$,
	\begin{equation*}
	\| e_A(t) \|_{L^2} - \| e_A(0)\|_{L^2} \leq C \int_0^t h^s\,\| \partial_t\,y(\s) \|_{H^s}\,\dd \s\,
	\end{equation*}
	 and use the differential equation $\partial_t\,y(t) = \frac{1}{2}\,\ii\,\Delta\,y(t)$, which yields
	\begin{equation*}
	\| e_A(t) \|_{L^2} \leq \| e_A(0)\|_{L^2}+ C \int_0^t h^{s}\,\| \Delta y(\s) \|_{H^{s}}\,\dd \s\,.
	\end{equation*}
	With $ e_A(0)=0 $, and taking the supremum over the integrand we finally obtain
	the $ L^2$-estimate~\eqref{eq:EAdiff-L2}.
	
	The $H^1$-estimate~\eqref{eq:EAdiff-H1} now follows directly from the inverse estimate, Theorem~\ref{thm:Brenner-Scottinv}.
\end{proof}

\begin{proof}[Proof of Proposition~\ref{thm:EAdiff-weak}]
	We start as in the proof of Proposition~\ref{thm:EAdiff}
	and recall~\eqref{eq:pt_e_e} in the form
	\begin{equation*}
	\partial_t\| e_A(t)\|_{L^2}^2 =
	2\,\Re\,(\Ph\,\partial_t y(t) - \partial_t y(t),e_A(t))_{L^2}\,.
	\end{equation*}
	By integrating over $\s$ from $0$ to $t$,
	applying partial integration and using H{\"o}lder's inequality we obtain
	\begin{align*}
	\| e_A(t) \|_{L^2}^2 - \| e_A(0)\|_{L^2}
	&= 2\,\int_0^t \Re(\Ph\,\partial_\s y(\s) - \partial_\s y(\s),e_A(\s))_{L^2}\,\dd\s
    \leq 2 \,\Re\,\int_{\Omega} \Big|\int_0^t \partial_\s (\Ph\,y(\s) -  y(\s))\,\overline{e_A(\s)}\,\dd\s\,\Big|\,\dd x \\
	& \leq 2\,\Re\int_{\Omega}  \Big|\Big[\big(\Ph\,y(\s)
    -  y(\s)\big)\,\overline{e_A(\s)}\Big]_0^t - \int_0^t(\Ph\,y(\s) - y(\s))\,\partial_\s\overline{e_A(\s)}\,\dd\s\,\Big|\,\dd x \\
	& \leq 2\,\Big(\|\Ph\,y(t) -  y(t)\|_{L^2}\,\| e_A(t)\|_{L^2}
	+ \int_{\Omega}\,\sup_{0 \leq \s \leq t} \big|\Ph\,y(\s) -  y(\s)\big|\,\cdot\,\Big|
	\int_0^t \partial_\s e_A(\s)\,\dd\s\,\Big|\,\dd x \Big) \\
	& \leq 2\,\Big(\|\Ph\,y(t) -  y(t)\|_{L^2}\,\| e_A(t)\|_{L^2} + \sup_{0\leq \s \leq t}\,\|\Ph\,y(\s) - y(\s)\big\|_{L^2}\,\| e_A(t)\|_{L^2} \Big) \\
	& \leq C\,h^{s} \sup_{0\leq \s \leq t}\,\| y(\s)\|_{H^s}\,\cdot\,\| e_A(t)\|_{L^2}\,,
	\end{align*}
	for $C$ depending on $d$ and $\Omega$.
	Since $e_A(0)=0$, dividing by $\| e_A(t)\|_{L^2}$ results in
	the $L^2$-estimate of Proposition~\ref{thm:EAdiff-weak}.
	
	The $H^1$-estimate now follows directly from the inverse estimate, Theorem~\ref{thm:Brenner-Scottinv}.
\end{proof}

The major part of the proof of Theorem~\ref{thm:convergence-complete}
relies on the following consistency result for the splitting operator~$\nS_h$.
\begin{theorem}
\label{thm:interpS-err}
Let $u\in H^{l}$ for $l\geq 1$, and \,$\| u \|_{H^{s}}\leq M_{s}$ for $s=\min\{l,p+1\}$.
Then, the difference $\nS_h(\t,\Ph\,u) -\Ph S(\t,u)$ is bounded by
\begin{equation*}
\begin{aligned}
\| \nS_h(\t,\Ph\,u) - \Ph\,\nS(\t,u) \|_{L^2} &\leq C\,\t\,(1+\tfrac{1}{\t})^\beta\, h^s\,,\\
\| \nS_h(\t,\Ph\,u) - \Ph\,\nS(\t,u) \|_{H^1} &\leq C\,\t\,(1+\tfrac{1}{\t})^\beta\,h^{s-1}\,,
\end{aligned}
\qquad
\,\beta=\max\{0,\sgn(p+3-l)\}\,,
\end{equation*}
where $C$ depends on $\|u\|_{H^{s+2(1-\beta)}}$, $d$ and $\Omega$.
\end{theorem}
\begin{proof}
We consider the specific errors of the subflows $\EA$ and $\EB$ and take account of the special structure of the Strang splitting operator.
The difference $\EBh(\t,\Ph\,u) - \Ph \EB(\t,u)$ is recast as a sum of terms with appropriate asymptotics in terms of $\t$ and $h$.
In the final step, we use the unitarity of $\EBh$ and $\ebh$. We abbreviate $w_h=\EAh (\half \t) \Ph\,u$ and $w=\EA(\half \t) u$ and obtain
\begin{subequations}
\label{eq:Sint}
\begin{align}
&\| \nS_h(\t,\Ph\,u) -\Ph\,\nS (\t,u) \|_{L^2} \notag \\
& \quad = \| \EAh (\half \t) \EBh(\t,w_h ) - \Ph \EA (\half \t) \EB(\t,w)\|_{L^2}\notag \\
& \quad \leq  \| \EAh (\half \t) \EBh(\t,w_h ) - \EAh (\half \t) \EBh(\t,\Ph w ) \|_{L^2} \notag \\
& \qquad {}	+ \| \EAh (\half \t) \EBh(\t,\Ph w ) - \EAh (\half \t)\Ph \EB(\t,w )\|_{L^2}\notag \\
 & \qquad {} + \| \EAh (\half \t)\Ph \EB(\t,w ) -\Ph \EA (\half \t) \EB(\t,w)\|_{L^2} \notag\\
& \quad \leq \| \ebht{w_h} w_h-\ebht{w_h} \Ph w \|_{L^2}  + \| \ebht{w_h}\Ph w- \ebht{\Ph w}\Ph w \|_{L^2}\notag \\
 & \qquad {} + \| \ebht{\Ph w}\Ph w- \ebt{w}\Ph w \|_{L^2}  + \| \ebt{w}\Ph w- \Ph\big(\ebt{w} w\big)\|_{L^2}\notag \\
 & \qquad {} + \| \EAh (\half \t) \Ph \EB(\t,w )-\Ph \EA (\half \t)\EB(\t,w)\|_{L^2} \notag\\
& \quad\leq \|w_h- \Ph w \|_{L^2}\label{eq:Sint-line1} \\
 & \quad {} + \| \ebht{w_h}\Ph w-\ebht{\Ph w_h} \Ph w \|_{L^2}\label{eq:Sint-line2} \\
 & \quad {} +\| \ebht{\Ph w}\Ph w- \ebt{w}\Ph w \|_{L^2} \label{eq:Sint-line3}\\
 & \quad {} +\| \ebt{w}\Ph w - \Ph\big(\ebt{w} w\big)\|_{L^2} \label{eq:Sint-line4}\\
 & \quad {} + \|\big( \EAh (\half \t)\Ph -\Ph \EA (\half \t)\big)\EB(\t,w )\|_{L^2} \,.\label{eq:Sint-line5}
\end{align}
\end{subequations}
Now, we consider the expressions~\eqref{eq:Sint} and obtain the following five estimates.
\begin{itemize}
	\item  For~\eqref{eq:Sint-line1} we use the Propositions~\ref{thm:EAdiff}
    and~\ref{thm:EAdiff-weak} and obtain
	\begin{equation*}
	\| w_h-\Ph  w \|_{L^2}
    \leq C\,\t\,h^{s}(\tfrac{1}{\t})^\beta\,\| u\|_{H^{s+2(1-\beta)}}\,,
	\end{equation*}
	where $s=\min\{l,p+1\}$ and $ \beta=\max\{0,\sgn(p+3-l)\}$.
	\item For~\eqref{eq:Sint-line2}, we use the linear variation-of-constant formula as in~\eqref{eq:B-B-diff}, with arguments $u=\Ph w$, $\wu=w_h$, and obtain
	\begin{align*}
	&\ebht{\Ph w}\Ph w- \ebht{w_h}\Ph w \\
	& \qquad {} =\int_0^{\t} \ebh(\t-\s,w_h)
              \big(\Bt_h(\Ph w)-  \Bt_h(w_h)\big) \ebh(\s,\Ph w) \Ph w\,\dd \s\,.
	\end{align*}
	Hence
	\begin{align*}
	&\| \ebht{w_h} \Ph w-\ebht{\Ph w}\Ph w \|_{L^2} \leq  \t \sup_{0\leq \s\leq\t} \big\|\big(\Bt_h(w_h ) - \Bt_h(\Ph w)\big)\EBh(\s,\Ph w)\big\|_{L^2}\,,
	\end{align*}
	since $\ebh$ is unitary.
	We now proceed similarly as in Appendix~\ref{sec:auxiliary}.
    By the same argument as in~\eqref{eq:Bt-linear} we obtain
	\begin{align*}
    \qquad
	\| (\Bt_h(w_h )- \Bt_h(\Ph w))\EBh(\s,\Ph w)\|_{L^2}
	&= \|\DPh(|w_h|^2 - |\Ph w|^2)\,\ebh(\s,\Ph w) \,\Ph w\|_{L^2} \\
	&\leq \|\DPh(|w_h - \Ph w|\,|w_h+ \Ph w|) \cdot \ebh(\s,\Ph w)\,\Ph w\|_{L^2}\,.
	\end{align*}
	Via Propositions~\ref{thm:DPh-L2-est} and~\ref{thm:EBh-H1-est}
    this can further be bounded by
	\begin{equation*}
	C\,\|w_h - \Ph w\|_{L^2}\,\|w_h+\Ph w\|_{H^1}\,\ee^{\s\, C\|u\|_{H^1}^2}\,\| u\|_{H^1}\,.
	\end{equation*}
	Now we use Propositions~\ref{thm:EAdiff} and~\ref{thm:EAdiff-weak}
    applied to the term $ \|w_h - \Ph w\|_{L^2}$
    in combination with the conservation properties
    of $\EA$ and $\EAh$ (see~\eqref{eq:semidiscrete-conservation-A-L2} and~\eqref{eq:discrete-conservation-A}) and obtain
	\begin{align}
	& \|w_h  - \Ph w\|_{L^2}\, \|w_h + \Ph w\|_{H^1}\,\ee^{\s\, C\|u\|_{H^1}^2}\,\| u\|_{H^1} \notag\\
	& \quad {} \leq C\,\t\,(\tfrac{1}{\t})^\beta\, h^{s}\,\| u\|_{H^{s+2(1-\beta)}} \big( \|\Ph\,u\|_{H^1} + \| \Ph w\|_{H^1} \big)\,\ee^{\s\, C\|u\|_{H^1}^2}\,\| u\|_{H^1}\,, \label{eq:wh-previous}
	\end{align}
	where $s=\min\{l,p+1\}$ and $ \beta=\max\{0,\sgn(p+3-l)\}$.
	
	The projection property~\eqref{eq:Proj-prop} and~\eqref{eq:wh-previous} yield
	\begin{equation}
	\begin{aligned}
	&\| \ebht{w_h} \Ph w -\ebht{\Ph w}\Ph w \|_{L^2}   \\
	&\qquad \leq C\,\t^2\,(\tfrac{1}{\t})^\beta\, h^{s}\,\|u\|_{H^{s+2(1-\beta)}}\big( \|\Ph\,u\|_{H^1} + \| \Ph w\|_{H^1} \big)\,\ee^{\t\, C\|u\|_{H^1}^2}\,\| u\|_{H^1}\\
	&\qquad \leq C\,\t^2\,(\tfrac{1}{\t})^\beta\, h^{s}\,\ee^{\t\, C\|u\|_{H^1}^2}\,\|u\|_{H^{s+2(1-\beta)}}\,\| u\|_{H^1}^2 \,.
	\end{aligned}
	\end{equation}
	\item
	For~\eqref{eq:Sint-line3} we use variation of constants as in~\eqref{eq:B-Bh-diff},
	\begin{align*}
	&\ebht{\Ph w}\Ph w-\ebt{ w}\Ph w  = \int_0^\t \ebh(\t-\s,\Ph w) \big(\Bt_h(\Ph w)- \Bt(w)\big)\eb(\s,w) \Ph w\,\dd\s \,.
	\end{align*}
	Hence, with $\Bt_h(w)=\Ph \Bt(w) $,
	\begin{align*}
	& \| \ebht{\Ph w} \Ph w - \ebt{w} \Ph w \|_{L^2}
	  \leq C \,\t \sup_{0 \leq \s \leq \t}\,\|( \Bt_h(\Ph w) - \Bt(w) ) \eb(\s,w) \Ph w \|_{L^2} \\
	& \qquad \leq C\,\,\t \sup_{0 \leq \s \leq \t}\big(\| (\Bt_h (\Ph w) - \Bt_h(w))\Ph w\|_{L^2}
	                           +\| (\Bt_h ( w) - \Bt(w))\Ph w\|_{L^2}\big) \\
	& \qquad = C\,\t \sup_{0 \leq \s \leq \t}\big( \| (\Bt_h (\Ph w) - \Bt_h(w))\Ph w\|_{L^2}
	                           + \| (\Ph(\Bt(w)) - \Bt(w))\Ph w\|_{L^2} \big)\,.
	\end{align*}
    Now we separately estimate the two contributions on the right-hand side.
	Analogously as for~\eqref{eq:Sint-line2}, we have
	\begin{align*}
	\| (\Bt_h (\Ph w) - \Bt_h(w))\Ph w\|_{L^2} & \leq C\,\| \Ph w - w\|_{L^2}\,\| \Ph w + w\|_{H^1}\,\|\Ph w\|_{H^1}\\
	& \leq C\,h^{s}\,\| w \|_{H^{s}}\big(\| w \|_{H^1} + \|\Ph w\|_{H^1}\big)\| \Ph w\|_{H^1}\\
	& \leq C\,h^{s}\,\| w \|_{H^{s}}\,\| w \|_{H^1}^2\,.
	\end{align*}
	For the second contribution
    we make use of an estimate based on Theorem~\ref{thm:Brenner-Scott1} and the Sobolev embeddings of $H^1$ in $L^4$ and $H^{s+1}$ in $W^{s}_4$,
	\begin{align*}
	\| (\Ph(\Bt ( w)) - \Bt(w))\Ph w\|_{L^2} & \leq \| \Ph(\Bt ( w)) - \Bt(w)\|_{L^4}\,\|\Ph w\|_{L^4}\\
    & \leq C\,h^{s}\,\| \Bt(w) \|_{W^{s}_4}\,\| \Ph w\|_{L^4} \\
    & \leq C\,h^{s}\,\| \Bt(w) \|_{H^{s+1}}\,\| \Ph w\|_{H^1} \\
	& \leq C\,h^{s}\,\| |w|^2 \|_{H^{s-1}}\,\| \Ph w \|_{H^1} \\
	& \leq C\,h^{s}\,\| w \|_{H^{\eta_1}}^2 \| w \|_{H^1}\,, \quad \eta_1=s - \chi_{[3,\infty]}(s)\,,
	\end{align*}
	where $\chi$ is the indicator function.
    The two values of $\eta_1$ are related to different bounds for $s\in\{1,2\}$, $\| u\, v \|_{H^{s-1}}\leq C\,\| u\|_{H^{s}}\,\|u \|_{H^{s}}$, while for higher values of $s$, the bounds are valid with $\| u\,v\|_{H^s} \leq C\,\|u \|_{H^s}\,\|v\|_{H^s}$.
	
	Altogether this yields
	\begin{align*}
	\| \ebht{\Ph w} \Ph w- \ebt{w} \Ph w\|_{L^2} \leq C \,\t\, h^{s}\,\big(\| w \|_{H^{\eta_1}}^2 \| w \|_{H^1}  + \| w \|_{H^{s}}\,\| w \|_{H^1}^2\big)
	\end{align*}
	with $ \eta_1=s - \chi_{[3,\infty]}(s) $.
	
	\item For~\eqref{eq:Sint-line4}, we apply variation of constants as for~\eqref{eq:Sint-line3} and obtain
	\begin{align*}
	& \| \ebt{w}\Ph w - \Ph(\ebt{w} w)\|_{L^2} \\
	& \quad \leq C\,\t\,\| \Bt(w)\Ph( \ebt{w} w)- \Ph ( \Bt(w)\ebt{w} w)\|_{L^2}\\
	& \quad\leq C\,\t\,\big(\| \Bt(w)\Ph( \ebt{w} w) - \Bt(w)\ebt{w} w\|_{L^2} \\
	& \quad\qquad\quad {} + \| \Bt(w)\ebt{w} w - \Ph ( \Bt(w)\ebt{w} w)\|_{L^2}\big)\\
	& \quad\leq C\,\t\,\big( \| \Bt(w)\|_{H^2} \| \Ph(\ebt{w} w )- \ebt{w} w\|_{L^2} \\
	& \quad\qquad\quad {} + \| \Bt(w)\ebt{w} w - \Ph ( \Bt(w)\ebt{w} w)\|_{L^2}\big)\\
	& \quad\leq C\,\t\,\big(  \| |w|^2\|_{L^2}\,h^{s}\,\| \ebt{w} w\|_{H^{s}} + h^{s}\,\| \Bt(w)\ebt{w} w \|_{H^{s}}\big)\\
	& \quad\leq C\,\t\,h^{s} \big(  \| w\|_{H^1}^2 \| \ebt{w} w\|_{H^{s}} + \| \Bt(w) \ebt{w} w\|_{H^{s}} \big)\,.
	\end{align*}
	Now we consider $\| \Bt(w)  \ebt{w} w\|_{H^{s}}$ in more detail for different values of $s$,
	\begin{align*}
	\| \Bt(w)  \ebt{w} w\|_{H^{s}} & \leq \| \DP |w|^2\|_{H^{\kappa}}\,\|  \ebt{w} w\|_{H^{s}}  \qquad \kappa = \max\{s,2\}\\
	& \leq \| w\|_{H^{\eta_2}}^2 \,\ee^{\t\, L_s} \| w\|_{H^{s}}\,, \qquad \eta_2 = s - \chi_{[2,\infty]}(s) - \chi_{[4,\infty]}(s)\,,
	\end{align*}
	with $L_s$ from~\eqref{eq:Hm-reg-bound} and where $\chi$ is the indicator function.
	Hence,
	\begin{equation*}
	\big\| \ebt{w} \Ph w - \Ph\big(\ebt{w} w\big)\big\|_{L^2}
	\leq C\,\t\,h^{s}\,\ee^{\t\, L_s}\,\| u\|_{H^s}\,\| u \|_{H^{\eta_2}}^2\,.
	\end{equation*}
	\item For~\eqref{eq:Sint-line5}, we use the results from Propositions~\ref{thm:EAdiff} and~\ref{thm:EAdiff-weak}, and obtain
	\begin{align*}
	\big\| \big( \EAh (\half \t) \Ph -\Ph\EA (\half \t)\big)\EB(\t,w ) \big\|_{L^2}& \leq C\,\t\,(\tfrac{1}{\t})^\beta\, h^{s} \sup_{0\leq \t_1\leq \t}\,\|\EA(\half \t_1) \EB(\t_1, w ) \|_{H^{s+2(1-\beta)}}\\
	& \leq C\,\t\,(\tfrac{1}{\t})^\beta\, h^{s} \sup_{0\leq \t_1\leq \t}\,\| \EB(\t_1, w) \|_{H^{s+2(1-\beta)}} \\
	& \leq  C\,\t\,(\tfrac{1}{\t})^\beta\, h^{s}\,\ee^{\t\, L_{s+2(1-\beta)}}\,\| u\|_{H^{s+2(1-\beta)}}\,,
	\end{align*}
	where $s=\min\{l,p+1\}$ and $ \beta=\max\{0,\sgn(p+3-l)\}$.
\item Combining these results, we obtain
\begin{align*}
\qquad
\begin{aligned}
\|\nS_h(\t,\Ph\,u)-\Ph\nS(\t,u)\|_{L^2}
&\leq C\,\t\, h^{s} \Big( (\tfrac{1}{\t})^\beta\|u\|_{H^{s+2(1-\beta)}} +\t\,(\tfrac{1}{\t})^\beta\,\ee^{\t\, C\|u\|_{H^1}^2}\,\|u\|_{H^{s+2(1-\beta)}}\,\| u\|_{H^1}^2\\
& \qquad\qquad {} +  \| u \|_{H^{\eta_1}}^2 \| u \|_{H^1}  + \| u \|_{H^{s}}\,\| u \|_{H^1}^2
 +   \ee^{\t\, L_s}\,\| u\|_{H^s}\,\| u \|_{H^{\eta_2}}^2 \\
& \qquad\qquad{} + (\tfrac{1}{\t})^\beta\,\ee^{\t\, L_{s+2(1-\beta)}}\,\| u\|_{H^{s+2(1-\beta)}} \Big) \\
& \leq C\,\t\, h^{s}\Big( \| u \|_{H^{s}}\,\|u\|_{H^{\eta_1}}\,\| u\|_{H^1}+ (\tfrac{1}{\t})^\beta\ee^{\t\,L_{s+2(1-\beta)}}\,\| u  \|_{H^{s+2(1-\beta)}}\\
& \qquad\qquad {} + \ee^{\t\,L_s}\,\| u \|_{H^s}\,\|u  \|_{H^{\eta_2}}^2 +\t\,(\tfrac{1}{\t})^\beta\,\ee^{\t\, C\|u\|_{H^1}^2}\,\|u\|_{H^{s+2(1-\beta)}}\,\| u\|_{H^1}^2\Big)\,,
\end{aligned}
\end{align*}
where $\eta_1=s - \chi_{[3,\infty]}(s)$ and $\eta_2 = s - \chi_{[2,\infty]}(s) - \chi_{[4,\infty]}(s)$.
Thus, we can find a constant $C^\ast$ for some $\t<t_n$ such that
\begin{equation}
\label{eq:S-Sh-estimate-Cast}
\|\nS_h(\t,\Ph\,u)-\Ph\nS(\t,u)\|_{L^2} \leq C^\ast\,\t\,(1+\tfrac{1}{\t})^\beta h^{s}\,.
\end{equation}
\end{itemize}
The $H^1$ approximation result follows directly from the $L^2$ approximation via the inverse estimate, Theorem~\ref{thm:Brenner-Scottinv},
\begin{equation*}
\|\nS_h(\t,\Ph\,u)-\Ph\nS(\t,u)\|_{H^1} \leq h^{-1}\,\|\nS_h(\t,\Ph\,u)-\Ph\nS(\t,u)\|_{L^2}\,,
\end{equation*}
which concludes the proof.
\end{proof}

\subsection{$ H^1 $-regularity of the fully discretized splitting operator} \label{sec:regularity-full}

\begin{lemma}
	\label{thm:Sh-regularity}
	Suppose that $u\in H^4$, $t_n \leq T$ is fixed and that $h$ is sufficiently small compared to $\| u\|_{H^4}$ and~$T$.
    Then we can bound the iterative application of the splitting operator $\nS_h$
    from~\eqref{eq:Strang-splitting-h} in $H^1$ in terms of $\wt{C}$ depending on $t_n$ and on $\| u \|_{H^4}$,
	\begin{equation*}
	\max_{1 \leq m \leq n}\|\nnSh{m}\Ph\,\nnS{n-m} u\|_{H^1} \leq \wt{C}\,.
	\end{equation*}
\end{lemma}

\begin{proof}
	We use induction over $n$ for $t_n=n\t \leq T $. For $n=1$
    we apply Proposition~\ref{thm:Sh-stability}, giving
	\begin{equation*}
	\| \nnSh{}\Ph\,u \|_{H^1} \leq \ee^{C\,\t a_1^2}\,\| u\|_{H^1}\,,
	\end{equation*}
	where $a_1=\|u \|_{H^1}$.	
	For $n=2$ we use the consistency estimate~\eqref{eq:S-Sh-estimate-Cast}, giving
	\begin{align*}
	\| \nnSh{2} \Ph\,u\|_{H^1} & = \| \nnSh{} \nnSh{} \Ph\,u\|_{H^1} \\
	& \leq \|\Ph\,\nnS{2} u\|_{H^1}  + \|(\Ph\,\nnS{}  - \nnSh{} \Ph)\,\nnS{} u \|_{H^1}
      + \| \nnSh{} (\Ph\,\nnS{} u- \nnSh{} \Ph\,u) \|_{H^1} \\
	& \leq \| \nnS{2} u\|_{H^1}+ C^\ast \,\t\,h + \ee^{C\,\t\, a_2^2}\,\| \Ph\,\nnS{} u - \nnSh{} \Ph\,u \|_{H^1} \\
	& \leq\| \nnS{2} u\|_{H^1} + C^\ast\,\t\,h\,\big( 1+ \ee^{C\,\t\, a_2^2}\big),
	\end{align*}
	where $a_2= \max\{\| \nnS{}u\|_{H^1}, \| \nnSh{} \Ph\,u\|_{H^1}\}$
    and where $C^\ast$ depends in particular on $\| \nnS{m}u \|_{H^{4}}$, $m\in\{0,1\}$
    as it appears in~\eqref{eq:S-Sh-estimate-Cast}.
    Here we have used the $H^1$-bound of Theorem~\ref{thm:interpS-err} for $s=2$, which involves the regularity requirement $u \in H^4$.
	For sufficiently small $h$ we can control the contribution of $\t\,( 1+ \ee^{C\,\t\, a_2^2})$ such that
	\begin{equation*}
	\| \nnSh{2} \Ph\,u\|_{H^1} \leq \| \nnS{2} u\|_{H^1} + C^\ast \leq \wt{C}\,,
	\end{equation*}
	for a constant $\wt{C}$ depending on $t_n$ and $\|u\|_{H^{4}}$. This follows from the regularity of the splitting solution in $H^{4}$,
	\begin{equation*}
	\| \nnS{2}u \|_{H^{4}}\leq \ee^{L_{4}\, t_2}\,\|u\|_{H^{4}} \leq \ee^{L_{4} T}\,\|u \|_{H^{4}}\leq \frac{\wt{C}}{2}\,,
	\end{equation*}
	with $L_4$ from~\eqref{eq:Hm-reg-bound}.
	
	\paragraph{$ \bullet $ $ n \mapsto n+1\,$}
    We assume inductively that $a_n= \max\limits_{m \in \{ 0,\ldots,n-1 \}}\| \nnSh{m}\Ph\,\nnS{n-1-m} u\|_{H^1} $ satisfies
    $ a_n \leq \wt{C} $ and show
	\begin{equation*}
	a_{n+1} \leq \wt{C}\,.
	\end{equation*}
	In particular, the constant $\wt{C}$ depends on $ T $ and $\| u\|_{H^{4}}$ such that
	\begin{equation*}
	\| \nnS{n} u\|_{H^4} \leq \ee^{L_{4} \,t_n}\,\|u \|_{H^{4}} \leq \frac{\wt{C}}{2}\,.
	\end{equation*}
	Now we use this inequality to show that $ \| \nnSh{n} \Ph\,u\|_{H^1} $ is bounded. In fact,
	\begin{align*}	
	\| \nnSh{n} \Ph\,u\|_{H^1} & \leq \big\| \Ph\,\nnS{n} u -
    \sum_{k=1}^{n}\big(\nnSh{k-1}  \Ph\,\nnS{} \nnS{n-k} u - \nnSh{k-1} \nnSh{}\Ph\,\nnS{n-k} u \big) \big\|_{H^1} \\
	& \leq \|\Ph\,\nnS{n} u\|_{H^1}  + \sum_{k=1}^{n}\,\|\nnSh{k-1}\Ph\,\nnS{} \nnS{n-k} u - \nnSh{k-1} \nnSh{}\Ph\,\nnS{n-k} u \|_{H^1}\\
	& \leq \| \nnS{n} u\|_{H^1} + \sum_{k=1}^n \ee^{C\,\t\,(k-1)\, a_n^2}\,\|\Ph\,\nnS{} \nnS{n-k} u -  \nnSh{}\Ph\,\nnS{n-k} u \|_{H^1}\\
	& \leq \| \nnS{n} u\|_{H^1} + \sum_{k=1}^n C^\ast\,\t\,h\,\ee^{C\,\t\,(k-1)\, a_n^2} ~ \leq  \| \nnS{n} u\|_{H^1} + C^\ast\, n\,\t\, h\,\ee^{C\,\t\,n\, a_n^2} \\
	& \leq \| \nnS{n} u\|_{H^1} + C^\ast\, T\, h\,\ee^{C\, T\, a_n^2} ~ \leq \| \nnS{n} u\|_{H^1} + C^\ast \leq \wt{C}
	\end{align*}
	for sufficiently small $h$ to control $t_n\ee^{C\, T\, a_n^2}$.
	
	Obviously, we can apply the same estimate for terms of the form
	\begin{equation*}
	\| \nnSh{m} \Ph\,\nnS{n-m} u\|_{H^1}\,, \quad m=0 \ldots n\,.
	\end{equation*}
	Hence
	\begin{equation*}
	a_{n+1} = \max_{m\in\{0,\ldots, n\}} \|  \nnSh{m}\Ph\,\nnS{n-m} u\|_{H^1} \leq \wt{C}\,.
	\end{equation*}
\end{proof}
\noindent This proof was inspired by~\cite{gauckler09}, where an Hermite spectral discretization
was considered.

\section{Implementation and numerical results}
\label{sec:num}

\subsection{Implementation aspects}
 For the efficient implementation of the FEM model introduced in Sec.~\ref{sec:setting},
 we use a method based on~\cite{Coh02} and~\cite{kormann2012time}.
 To this end we choose (tensor) Gauss--Lobatto nodes of degree $p$ on rectangular elements for the definition of the nodal basis and for the numerical evaluation of the inner products in~\eqref{eq:Matrix-definition}. These nodes allow exact integration of polynomials up to degree $2p-1$, hence the evaluation of the matrix $K$, which involves the gradients $\nabla v_{(k,j)}$, is exact. The evaluation of the matrix $M$ involves integrals of the form
\begin{equation*}
\int_{\Omega_k} v_{(k,i)}(x)\,v_{(k,j)}(x)\,\dd x \approx
\sum_\ell w_\ell\,v_{(k,i)}(x_\ell)\,v_{(k,j)}(x_\ell) =
\sum_\ell w_\ell\,\delta_{i{}\ell}\,\delta_{j{}\ell} = w_i\,\delta_{i{}j}\,,
\end{equation*}
where $w_l$ are the associated quadrature weights. Hence the matrix $M$ is diagonal,
and $M^{-1}K$ preserves the sparsity of $K$ (see Algorithm~\ref{alg:SPE-alg}), and likewise for the matrix $\potm$.

Analogously, the evaluation of $F$ simplifies to
\begin{equation*}
F(c) = M\cdot \big( c\;{.*}\;\ol{c}\big)\,,
\end{equation*}
where ${.*}$ denotes component-wise multiplication.
For the computation of the numerical solution $\psi_n=\nS_h^n(\t,\ini)$ for the full FEM discretization, we refer to Algorithm \ref{alg:SPE-alg}.
\begin{algorithm}
\caption{A splitting solution of the Schr\"odinger-Poisson problem}
\label{alg:SPE-alg}
\begin{algorithmic}[1]
\Procedure {Schr\"odinger-Poisson}{$M,K,\ini$}
\State $\ini\ldots$ interpolated initial function
\Procedure {Calculate matrices}{$v_i$} \Comment {$v_i\ldots$Galerkin basis functions}
\For {$k=1\ldots n_G$} \Comment{$n_G\ldots$ number of integration nodes per element}
\State $K_{i,j}:=K_{i,j} + w_k\cdot\nabla v_i(x_k)\cdot\nabla v_j(x_k)\cdot \det(T_k)$
\Comment{$T_k\ldots$ Jacobian of translation to $x_k$}
\State $M_{i,j}:=M_{i,j} + w_k\cdot v_i(x_k)\cdot v_j(x_k)\cdot \det(T_k)$			
\Comment{$w_k\ldots$ integration weights}
\State Preliminary calculation of $\potm(\cdot)$, $F(\cdot)$
\State $\potm_{i,j}:=\potm_{i,j} + w_k\cdot v_m(x_k)\cdot v_i(x_k)\cdot v_j(x_k)\cdot \det(T_k)$ \Comment{$v_m\, v_i=v_i^2$ using $v_m(x_k)v_i(x_k)=\delta_{m,i}\delta_{m,k}$}
\State $F_{i}:=F_{i} + w_k\cdot v_i(x_k)^2\cdot v_i(x_k)\cdot \det(T_k)$
\EndFor
\EndProcedure
\State $t:=\t$			\Comment{$\t\ldots$initial time stepsize}
\State $\psi:=\ini$
\While {$t<T$}
\State	$\psi_{\text{tmp}}:=\psi_\t(\t,\psi,\alpha,\beta)$   	\Comment{$\alpha$, $\beta$\ldots splitting coefficients}
\State	$\texttt{err}:=$ error estimator$(\psi_{\text{tmp}},\t,\psi)$   \Comment{any suitable error estimator}
\If {$\texttt{err}<\texttt{tol}$}
\State	$\psi:=\psi_{\text{tmp}}$, $t:=t+\t$
\Else
\State Choose smaller $\t$ to reduce \texttt{err}
\EndIf
\EndWhile
\EndProcedure
\Statex
\Procedure {$\psi_{\t}$}{$\t$,$\psi_{\texttt{inp}}$,$\alpha$,$\beta$}
\State $\psi_\t:=\psi_{\texttt{inp}}$
\For {$k=1\ldots n_s$} 	\Comment{$n_s\ldots$ number of splitting stages}
\State		$\psi_{\t}=\EBh(\beta(k) \cdot\t,\EAh(\alpha(k)\cdot \t,\psi_\t))$
\EndFor
\EndProcedure
\Statex
\Procedure {$\EAh$}{$\t$, $\psi_{\texttt{inp}}$}
\State $A:=-\ii \,\half \,M^{-1}\cdot K$					\Comment{$A\ldots$temporary matrix}
\State $\EAh:=\ee^{\t A}\cdot \psi_{\texttt{inp}}$ 	\Comment{matrix exponential calculated via an effective solver (\texttt{expokit})}
\EndProcedure
\Statex
\Procedure {$\EBh$}{$\t$, $\psi_{\texttt{inp}}$}
\State $d:=-K^{-1}\cdot F\cdot |\psi_{\texttt{inp}}|^2$ 	\Comment{$d\ldots$ solution of Poisson problem} \label{line:c}
\State $\hat{B}:=-\ii\,\,M^{-1}\cdot \potm\cdot d$			\Comment{$\hat{B}\ldots$ diagonal matrix} \label{line:d}
\State $\EBh:=\ee^{\t \hat{B}}\cdot \psi_{\texttt{inp}}$			\Comment{exponential calculated via pointwise multiplication}
\EndProcedure
\end{algorithmic}
\end{algorithm}

The obtained systems of differential equations for $\EAh$ can be solved efficiently via fast exponential solvers (for instance the function \texttt{expv} from the package \texttt{expokit}, see~\cite{sidje98}, which is based on an adaptive Krylov integrator, see~\cite{saad92s}), and since $K$ is a symmetric positive definite band matrix, the Poisson problem can be solved efficiently by common solvers for sparse systems of linear equations.

In Algorithm~\ref{alg:SPE-alg} we have indicated a time-adaptive version based on an appropriate local error estimator.
For this purpose, one may e.g. adopt the approach from~\cite{auzingeretal13b}.
For adaptivity in space, an appropriate a posteriori error estimator is required, but this is not in the scope of this presentation.

\subsection{Numerical example}
\label{sec:num-example}
We illustrate the performance of time-splitting for a two-dimensional test example. The problem data are chosen as follows:
\begin{itemize}
\item $\Omega=[0,5]^2$
\item $\ini(x,y)=10\,\ee^{-10((x-2.5)^2+(y-2.5)^2)}$ (Gaussian initial state)
\item Integration from $t=0$ to $t=0.1$.
In Figure~\ref{Fig:Plot} we display the wave function at time $t=0.1$ using a $100\times100$ mesh
and polynomial basis functions of degree $2$,
obtained via a fourth order splitting method with time stepsize $\t=0.0005$.
\end{itemize}
\paragraph{Global time-splitting error} For the finite element discretization we choose $25\times25$ uniform
rectangular elements of degree $p=2$ with Gauss--Lobatto nodes.
We apply time-splitting methods of orders $q=1$ to $4$, namely
Lie--Trotter splitting ($q=1$), Strang splitting ($q=2$),
a scheme of order $q=3 $ with rational coefficients by Ruth
(\cite[3rd order scheme from the pair \text{\tt Emb 3/2 RA}]{splithp}),
and an optimized scheme of order $ q=4 $ by Blanes and Moan
(\cite[4th order scheme from the pair \text{\tt Emb 4/3 BM PRK/A}]{splithp});
see the collection~\cite{splithp} for tables of coefficients and further references.

In Figure~\ref{Fig:t-errors_6methods} we display the $L^2$-norm of the global error at $t_n=0.1$ for different choices of the time stepsize $\t$ together with the observed orders $\hat{q}$,
\begin{equation*}
\| \psi(t_n) - \psi_n \|_{L^2} \approx C\,\t^{\hat{q}}\,.
\end{equation*}
A reference solution was obtained using a high order splitting scheme with a significantly refined time stepsize.

\paragraph{FEM approximation error} In Figure~\ref{Fig:h-errors_4methods-full}
we document the behavior of the spatial discretization error using the Strang splitting method for a fixed time stepsize $\t=0.002$ in dependence of the FEM-mesh for different values of the polynomial degree $p$, varying the mesh parameter $h$ from $1$ to $2^{-6}$ and determining the respective observed order $\hat{p}$ of the spatial error via extrapolation for $h\rightarrow 0$,
\begin{align*}
\| \psi(t_n) - \psi_n \|_{L^2} \approx C\,h^{\hat{p}+1}\,.
\end{align*}

   \begin{figure}[h!]
   	\centering
   	\includegraphics[width=0.45\textwidth]{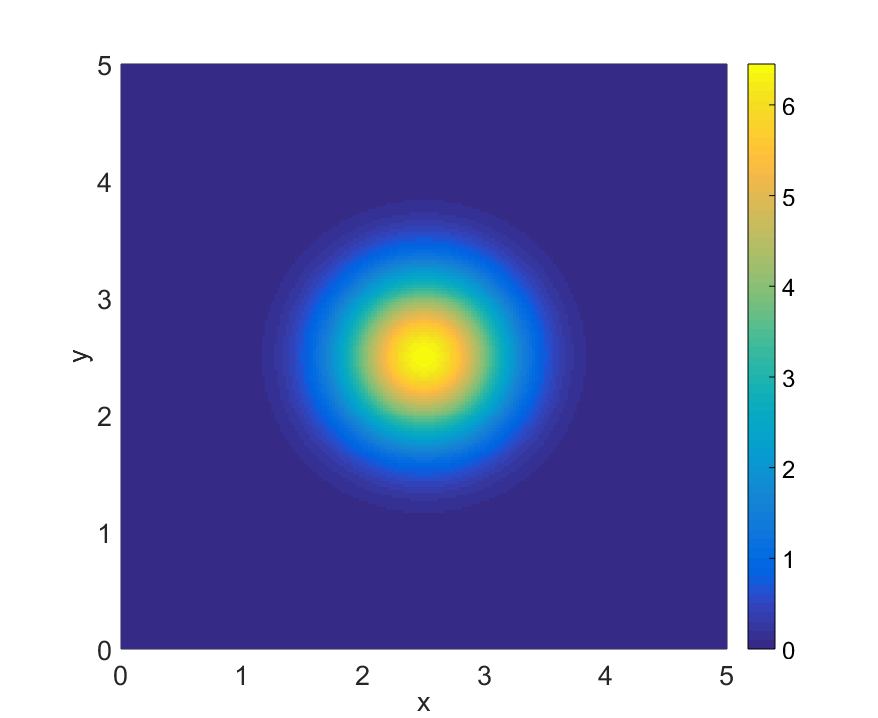}
   	\includegraphics[width=0.45\textwidth]{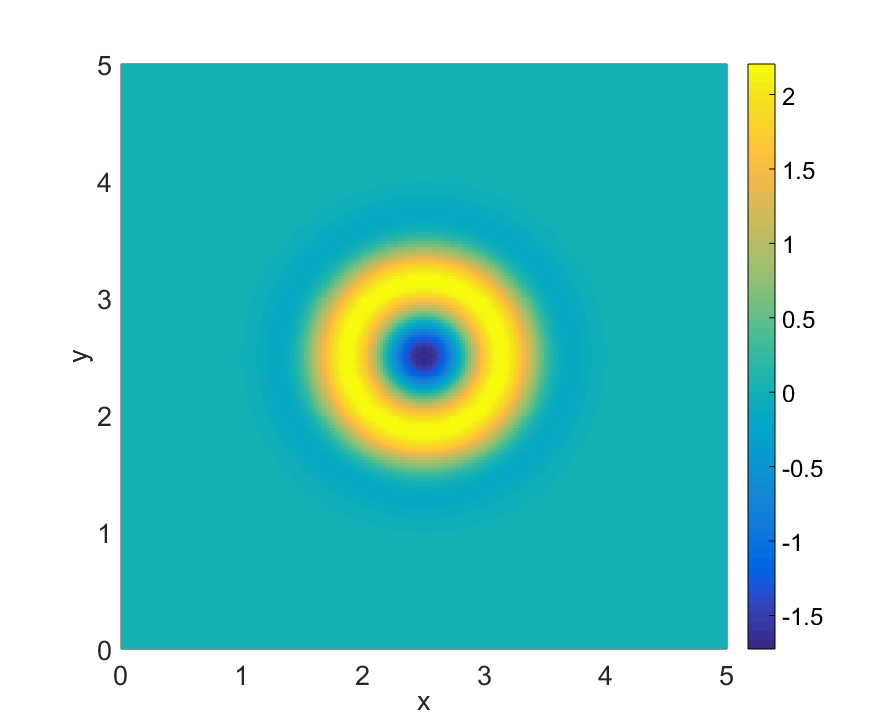}
   	\caption{Wave function at $t=0.1$. Left: Absolute value $|\psi|$. Right: Real part $\Re\,\psi$.}
   	\label{Fig:Plot}
   \end{figure}

\begin{figure}[h!]
    	\centering
        \includegraphics[width=0.45\textwidth]{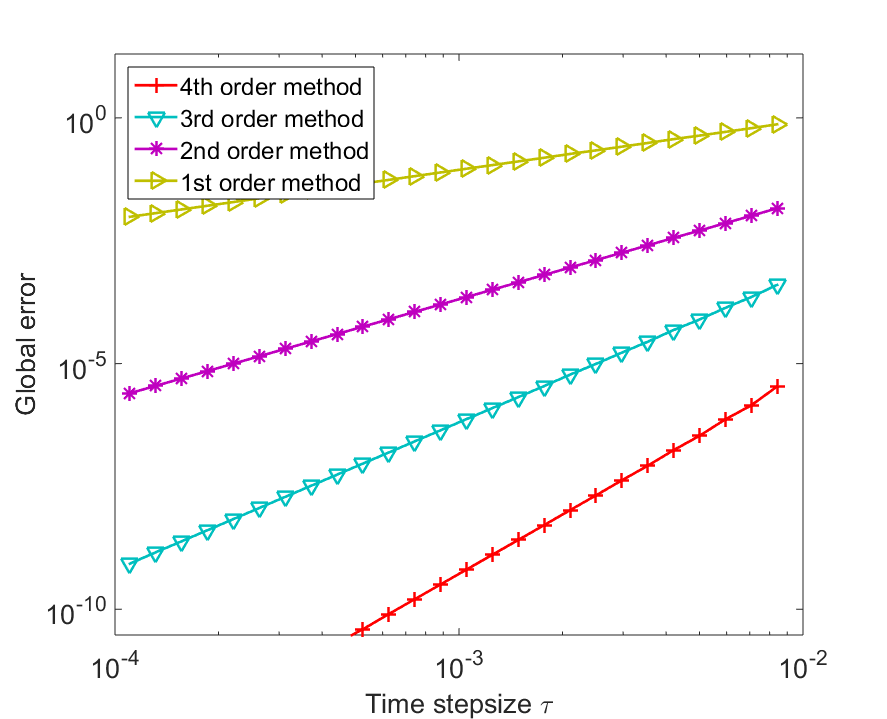}
        \includegraphics[width=0.45\textwidth]{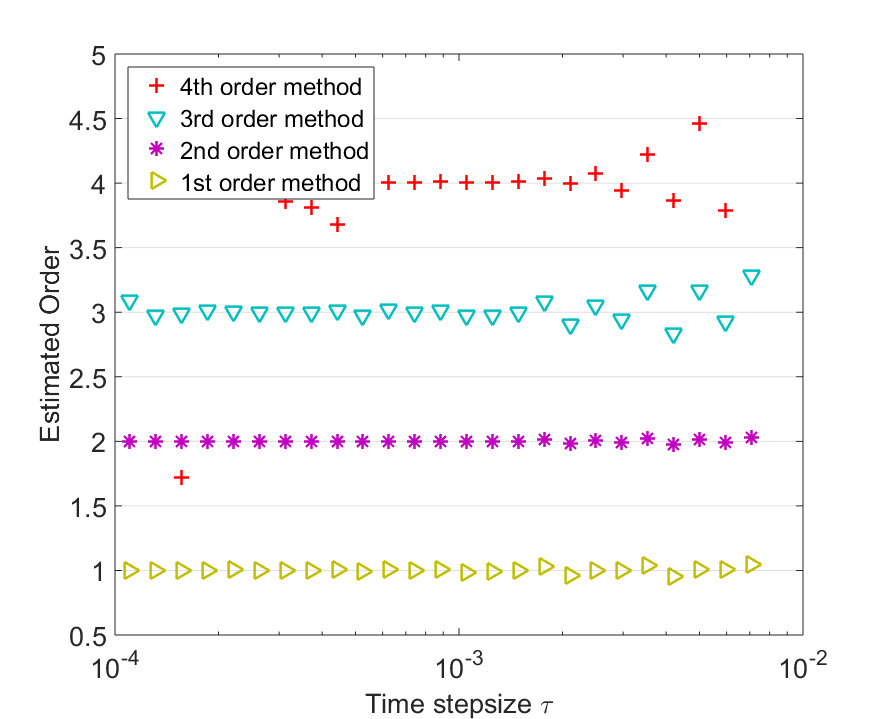}
    	\caption{\textbf{Global time-splitting error.}
        Left: Global time-splitting error in $ L^2 $ at $t=0.1$ in dependence of the time stepsize $\t$
        for splitting methods orders $q=1$ to $4$. Right: Observed order~$\hat{q}$.}
    \label{Fig:t-errors_6methods}
    \end{figure}

\begin{figure}[h!]
    	\centering
        \includegraphics[width=0.45\textwidth]{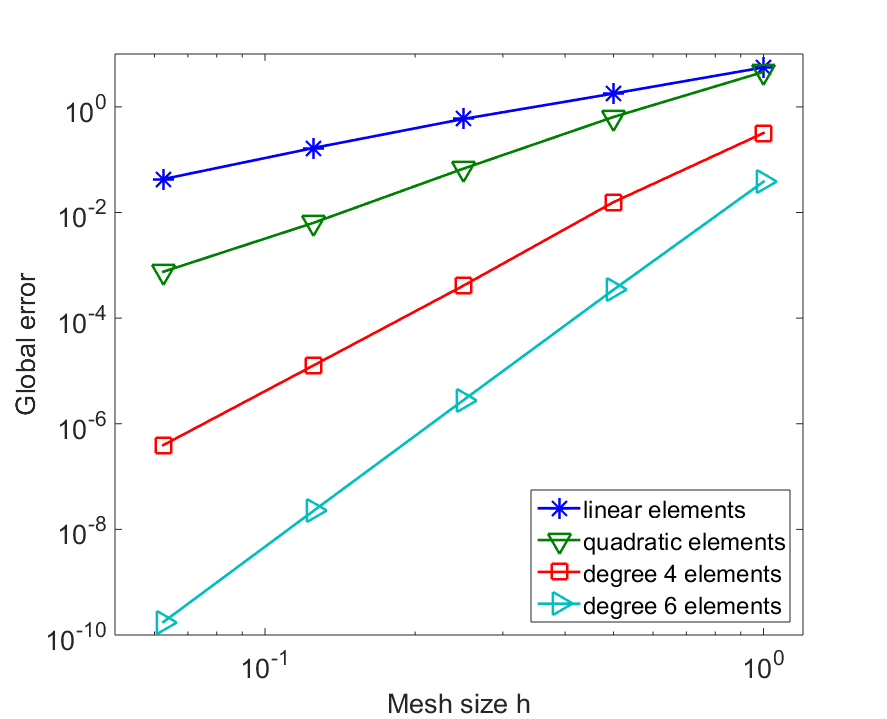}
        \includegraphics[width=0.45\textwidth]{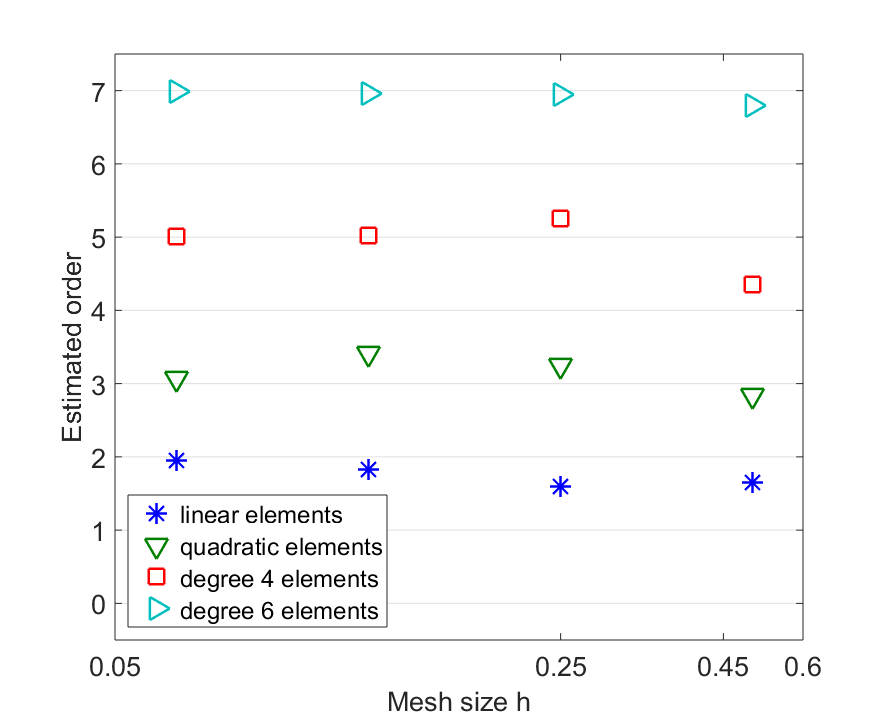}
    	\caption{\textbf{Spatial approximation order.} Left: FEM approximation error in $ L^2 $ at $t=0.1$
        in dependence of the meshsize $h$ for polynomial degree $p\in\{1,2,4,6\}$. Right: Observed order $\hat{p}+1$.}
    \label{Fig:h-errors_4methods-full}
    \end{figure}
\appendix
\section{Solution representation by variation-of-constant formulas}
\label{sec:VOC-appl}
Since in our context nonlinear operators $B$ and $B_h$ arise,
we will resort to the following variants of the variation-of-constants formula.
We recall that $B$ and $\hat{B}$ are defined in~\eqref{eq:AB} and define subflows~$\EB$ and $\eb$ in the subproblem~\eqref{eq:EB-problem} and~\eqref{eq:eb-problem-linear}.
In the spatially discrete case, the operators $B_h$ and $\hat{B}_h$ are defined in~\eqref{eq:AhBh}
and define the subproblems~\eqref{eq:EBh-weak-problem} and~\eqref{eq:ebh-weak-problem} for the evolutionary operators $\EBh$ and $\ebh$.
\begin{enumerate}[(i)]
\item \textbf{Same operators $B_h$, but with different argument.}
	
To rewrite the difference
\begin{align*}
\EBh(\t,u) - \ebh(\t,\wu)u &= \underbrace{\ebh(\t,u)u}_{v(\t)} - \underbrace{\ebh(\t,\wu)u}_{w(\t)}\,,
\end{align*}
we use the fact that for all $\phi_h\in \nVh$,
\begin{align*}
\big(v'(\t),\phi_h\big)_{L^2} &=\big( \Bt_h(u) \,v(\t), \phi_h\big)_{L^2}\,, \qquad \big(w'(\t),\phi_h\big)_{L^2} = \big(\Bt_h(\wu)\, w(\t),\phi_h\big)_{L^2}\,, \qquad v(0)=w(0)\,,
\end{align*}
which defines a new differential equation,
\begin{align*}
\begin{cases}
~\big((v-w)'(\t),\phi_h\big)_{L^2} = \big(\Bt_h(\wu) \big(v-w\big)(\t) + \big( \Bt_h(u) - \Bt_h(\wu)\big) v(\t),\phi_h\big)_{L^2}\,,\\
~\big((v-w)(0),\phi_h\big)_{L^2} = 0\,.
\end{cases}
\end{align*}
By the variation-of-constant formula we obtain the mild formulation
\begin{equation}
\label{eq:B-B-diff}
\big(v - w)(\t)\,u = \int_0^\t \ebh(\t-\s,\wu)\,\big( \Bt_h(u) - \Bt_h(\wu)\big)\,\ebh(\s,u)\,u\,\dd\s\,.
\end{equation}
\item \textbf{Different operators $B$, $B_h$}.

To rewrite the difference
\begin{equation*}
\ebt{ \wu}u  - \EBh(\t,u)
= \underbrace{\eb(\t,\wu)u}_{v(\t)} - \underbrace{\ebh(\t,u)u}_{w(\t)}\,,
\end{equation*}
we use the fact that $\nVh\subset \nV$, such that for all $\phi_h\in\nVh$,
\begin{equation*}
\big(v'(\t),\phi_h\big)_{L^2}
= \big(\Bt(\wt u) \,v(\t),\phi_h\big)_{L^2}\,, \quad \big( w'(\t),\phi_h\big)_{L^2}
= \big(\Bt_h(u)\, w(\t),\phi_h\big)_{L^2}\,, \quad v(0)=w(0)\,.
\end{equation*}
Again we obtain a differential equation,
\begin{equation*}
\begin{cases}
~\big((v-w)'(\t),\phi_h\big)_{L^2} = \big(\Bt_h(u) \big(v-w\big)(\t) + \big( \Bt(\wt u) - \Bt_h(u)\big) v(\t),\phi_h\big)_{L^2}\,,\\
~\big((v-w)(0),\phi_h\big)_{L^2} = 0\,
\end{cases}
\end{equation*}
such that the variation-of-constant formula yields
\begin{equation}
\label{eq:B-Bh-diff}
\big(v - w)(\t)\,u = \int_0^\t \ebh(\t-\s,u)\,\big( \Bt(\wt u) - \Bt_h(u)\big)\,\eb(\s,\wu)\,u \,\dd\s\,.
\end{equation}
\end{enumerate}

\section{Useful inequalities}\label{sec:sobolev}

In our theoretical estimates, we recurrently resort to estimates of Sobolev type.
For convenience of the reader, we briefly recapitulate these technical tools here.
We start by repeating some elementary notions from functional
analysis, see for example~\cite{brennerscott02,hackbusch92,miklavcic98}.
The underlying space is $L^2$ equipped with the inner product $(\cdot , \cdot )_{L^2}$,
\begin{equation*}
(v,w)_{L^2} = \int_{\Omega} v(x)\,\overline{w(x)} \; \dd x\,, \qquad v,w \in L^2\,,
\end{equation*}
and the norm $\|\,\cdot\,\|_{L^2}$,
where $\Omega$ is a bounded domain with smooth boundary (for
the Sobolev embeddings cited below, it is necessary that $\Omega$ satisfies
the cone condition).

The set of all functions in $L^2$ having weak derivatives up to
order $\leq k$ is denoted as the \emph{Sobolev space} $H^k$.
It is equipped with the norm
\[ \|u\|_{H^k}:=\Big(\sum_\alpha \| \partial^\alpha u\|_{L^2}^2\Big)^{1/2},\]
where the sum is over all derivatives up to order $k$.

Furthermore, we will denote by $\|\,\cdot\|_{L^\infty}$ the supremum norm
on the space of functions bounded almost everywhere.

In our analysis, we will make use of
the following results, see for instance~\cite{brennerscott02}.
Our formulations are specific to $\RR^d$, $d\in\{2,3\}$:
\begin{theorem}\label{thm:sobolev1}
Let $k, m \in \NN$ such that $k-m>3/2$. Then for $u\in H^k$ there is a $C^m$ function in
the $L^2$ equivalence class of $u$ and
\begin{equation*}
\|u\|_{C^m}:=\sum \|\partial^\alpha u\|_{L^\infty}\leq \const\|u\|_{H^k},
\end{equation*}
where the sum is over all derivatives of order up to $m$.
\end{theorem}

This implies the following inequalities, see for instance
\cite{adams75},~\cite{brennerscott02},~\cite{harlit34}, and~\cite{miklavcic98}:
\begin{corollary}\label{cor:sobolev2}
For $u,v \in H^2$, the following inequalities hold:
\begin{align*}
 \|uv\|_{L^2}&\leq \|u\|_{L^2}\|v\|_{L^\infty}\leq \const\, \|u\|_{L^2}\|v\|_{H^2}\,,\\
 \|uv\|_{H^1}&\leq \const\,\|u\|_{H^1}\|v\|_{H^2}\,,\\
 \|uv\|_{H^2}&\leq \const\,\|u\|_{H^2}\|v\|_{H^2}\,,\\
 \|uv\|_{L^2}&\leq \const\,\|u\|_{L^4}\|v\|_{L^4} \leq
\const\|u\|_{H^1}\|v\|_{H^1}\,,\\
 \|uvw\|_{L^2}&\leq \const\|u\|_{L^6}\|v\|_{L^6}\|w\|_{L^6} \leq
\const\|u\|_{H^1}\|v\|_{H^1}\|w\|_{H^1}\,.
\end{align*}
\end{corollary}

\section{Auxiliary results} \label{sec:auxiliary}
{\revblue This section contains a collection of useful theorems and bounds which are used in the convergence theory in~Section~\ref{sec:conv-ana}.

\subsection{Conservation and stability properties of the subflows}
\begin{proposition}
\label{thm:semidiscrete-conservation} $~$
\begin{enumerate}[(i)]
\item
The evolution operator $ \EA(t) $ is unitary with respect to $ \norm{L^2}{\,\cdot\,} $ and $ \norm{H^1}{\,\cdot\,} $, for $ t>0 $
and $ u \in H_0^1$,
\begin{subequations}
\label{eq:semidiscrete-conservation-A}
\begin{align}
\norm{L^2}{\EA(t)\,u} &= \norm{L^2}{u}\,, \label{eq:semidiscrete-conservation-A-L2} \\
\norm{H^1}{\EA(t)\,u} &= \norm{H^1}{u}\,. \label{eq:semidiscrete-conservation-A-H1}
\end{align}
\end{subequations}
\item
The evolution operator $ \EB(t,\,\cdot\,) $ is unitary with respect to $ \norm{L^2}{\,\cdot\,} $
for $ t>0 $ and $ u \in H_0^1$,
\begin{equation}
\label{eq:semidiscrete-conservation-B}
\norm{L^2}{\EB(t)\,u} = \norm{L^2}{u}\,.
\end{equation}
\end{enumerate}
\end{proposition}

\begin{proof}
\begin{enumerate}[(i)]
\item
We proceed from {the weak formulation}~\eqref{eq:EB-weak-problem},
{
\begin{equation}
\label{eq:EB-weak-problem2}
\begin{cases}
~\scp{L^2}{\partial_t \psi}{\phi} = -\,\ii\,\scp{L^2}{\pot\,\psi}{\phi}
  \quad \text{for all} ~~ \phi \in H^1_0\,,  \\
~\psi\big|_{t=0} = u\,,
\end{cases}
\end{equation}
}
set $ \phi=\psi=\EA(t,u) $,
\begin{equation*}
\scp{L^2}{\partial_t \psi}{\psi} = -\,\tfrac{1}{2}\,\ii\,\scp{L^2}{\nabla\psi}{\nabla\psi}\,,
\end{equation*}
and obtain
\begin{equation*}
\partial_t\| \psi \|_{L^2}^2 = \partial_t(\psi,\psi)_{L^2}
= 2\,\Re(\partial_t \psi,\psi)_{L^2} = \Re(-\ii\,\scp{L^2}{\nabla\psi}{\nabla\psi}) = 0\,,
\end{equation*}
which implies~\eqref{eq:semidiscrete-conservation-A-L2}.
Furthermore, setting $ \phi=\partial_t \psi = \partial_t\,\EA(t,u) $ in~\eqref{eq:EB-weak-problem2},
\begin{equation*}
\norm{L^2}{\partial_t \psi} = \scp{L^2}{\partial_t\psi}{\partial_t \psi} = -\,\tfrac{1}{2}\,\ii\,\scp{L^2}{\nabla\psi}{\nabla\partial_t\psi}\,,
\end{equation*}
we obtain
\begin{equation*}
\partial_t |\psi|_{H^1}^2
= \partial_t(\nabla \psi,\nabla \psi)_{L^2}
= 2\,\Re\scp{L^2}{\nabla\psi}{\nabla\partial_t\psi}
= 2\,\Re(2\,\ii\,\norm{L^2}{\partial_t \psi}^2) = 0\,,
\end{equation*}
which together with~\eqref{eq:semidiscrete-conservation-A-L2} implies~\eqref{eq:semidiscrete-conservation-A-H1}.
\item For the flow defined by~\eqref{eq:eb-problem-linear},
{$\partial_t \psi = \Bt(w)\psi, \ \psi\big|_{t=0} = u,$}
 and $\psi=\eb(t,w)\,u$ we have
\begin{equation*}
\partial_t \| \psi \|_{L^2}^2 = 2\, \Re\,( \psi, \Bt(w)\, \psi)_{L^2} =  2\, \Re\,( |\psi|^2, \Bt(w))_{L^2} = 0,
\end{equation*}
which, in particular, implies~\eqref{eq:semidiscrete-conservation-B}.
\end{enumerate}
\end{proof}

\begin{remark} \label{rem:Hk-cons}
{\rm{
More generally, the $H^k$-norms for $k\geq1$ are conserved under the flow $ \EA(t) $. To see this, we consider
the strong formulation~\eqref{eq:EA-problem}, {$\partial_t\psi=A\psi,\ \psi(0)=u,$} with $u\in H^{k+2}\cap H_0^1$ and wish
to show that $\|\partial_x^k \psi \|_{L^2} =\|\partial_x^k u \|_{L^2}$ for
any partial derivative~$\partial_x$. We compute
\begin{equation*}
\partial_t (\partial_x^k \psi, \partial_x^k \psi)_{L^2}
= 2\,\Re(\partial_x^k \partial_t \psi,\partial_x^k \psi)_{L^2}
= 2\,\Re\big(\tfrac{1}{2}\,\ii\,(\partial_x^{k+2}\psi,\partial_x^k\psi)_{L^2}\big)
= \Re\big(-\ii\,\|\partial_x^{k+1} \psi\|_{L^2}^2 \big)= 0\,.
\end{equation*}
Via a density argument, the result also holds for all $\psi \in H^k \cap H_0^1 $.
}}
\end{remark}

\subsection{Conservation and stability properties of the discrete subflows}
For our convergence analysis we will make use of the following facts.
\begin{proposition}
\label{thm:discrete-conservation} $~$
\begin{enumerate}[(i)]
\item
The evolution operator $ \EAh(t) $ is unitary with respect to $ \norm{L^2}{\,\cdot\,} $ and $ \norm{H^1}{\,\cdot\,} $, for $ t>0 $
and $ u_h \in \nVh$,
\begin{subequations}
\label{eq:discrete-conservation-A}
\begin{align}
\norm{L^2}{\EAh(t)\,u_h} &= \norm{L^2}{u_h}\,, \label{eq:discrete-conservation-A-L2} \\
\norm{H^1}{\EAh(t)\,u_h} &= \norm{H^1}{u_h}\,. \label{eq:discrete-conservation-A-H1}
\end{align}
\end{subequations}
\item
The evolution operator $ \ebh(t,\,\cdot\,) $ is unitary
with respect to $ \norm{L^2}{\,\cdot\,} $ i.e., for $ t>0 $ and $u_h,\, w_h \in \nVh $,
\begin{subequations}
\label{eq:discrete-conservation-B}
\begin{equation}
\label{eq:discrete-conservation-B-L2}
\norm{L^2}{\ebh(t,w_h) u_h} = \norm{L^2}{u_h}\,.
\end{equation}
For $w_h=u_h$, this implies $\norm{L^2}{\EBh(t,u_h)} = \norm{L^2}{u_h}$.
Furthermore, $ \ebh(t,\,\cdot\,)$ satisfies the differential inequality
\begin{equation}
\label{eq:H1-EBh-property}
\partial_t |\ebh(t,w_h) u_h|_{H^1} 
                                \leq \norm{L^2}{\ebh(t,w_h) u_h \nabla\pot_h}\,,
\end{equation}
where $\pot_h=\DPh(|w_h|^2)$.
\end{subequations}
\end{enumerate}
\end{proposition}
\begin{proof}
\begin{enumerate}[(i)]
\item
For $ u_h \in \nVh $ and $ \psi_h = \EAh(t)\,u_h\in \nVh $ we have
$ \partial_t \psi_h = \half\,\ii\,\Delta_h \psi_h $. Hence by
{the definition of the discrete Laplacian}~\eqref{eq:Deltah-Def},
\begin{equation*}
\partial_t\| \psi_h \|_{L^2}^2 = \partial_t(\psi_h,\psi_h)_{L^2}
= 2\,\Re(\partial_t \psi_h,\psi_h)_{L^2} = \Re(\ii\,\Delta_h \psi_h,\psi_h)_{L^2}
= -\Re(\ii\,\nabla \psi_h,\nabla \psi_h)_{L^2} = 0\,,
\end{equation*}
which implies~\eqref{eq:discrete-conservation-A-L2}. Furthermore,
\begin{align*}
\partial_t |\psi_h|_{H^1}^2
&= \partial_t(\nabla \psi_h,\nabla \psi_h)_{L^2}
= 2\,\Re(\nabla\,\partial_t \psi_h,\nabla \psi_h)_{L^2}
=- 2\,\Re(\partial_t \psi_h,\Delta_h \psi_h)_{L^2}\\
& =-2\, \Re(\partial_t \psi_h,-\,2\,\ii\, \partial_t \psi_h)_{L^2}
= -4\,\Im(\partial_t \psi_h,\partial_t \psi_h)_{L^2}
 = 0\,,
\end{align*}
which implies~\eqref{eq:discrete-conservation-A-H1}.
\item For $u_h,\,w_h \in \nVh $ and $ \psi_h = \ebh(t,w_h)\,u_h $ we have
$ \partial_t \psi_h = -\ii\,\pot_h \psi_h $ with $ \pot_h = \DPh(|w_h|^2) $, hence
\begin{equation*}
\partial_t\,\norm{L^2}{\psi_h}^2
= \partial_t \scp{L^2}{\psi_h}{\psi_h}
= 2\,\Re\scp{L^2}{\partial_t \psi_h}{\psi_h}
= 2\,\Im\scp{L^2}{\pot_h\,\psi_h}{\psi_h} = 0,
\end{equation*}
since $ \pot_h $ is real. This implies~\eqref{eq:discrete-conservation-B-L2}.

On the other hand, $\ebh$ does not conserve the $H^1$-norm.
To derive a bound we compute
\begin{equation*}
\begin{split}
\partial_t |\psi_h|_{H^1}^2
&= \partial_t \scp{L^2}{\nabla\psi_h}{\nabla\psi_h}
 = 2\,\Re\scp{L^2}{\nabla\,\partial_t \psi_h}{\nabla\psi_h}
 = 2\,\Re(-\,\ii\,\scp{L^2}{\nabla(\pot_h \psi_h)}{\nabla\psi_h}) \\
&= 2\,\Im\scp{L^2}{\pot_h \nabla\psi_h}{\nabla\psi_h}
    + 2\,\Im \scp{L^2}{\psi_h \nabla\pot_h}{\nabla\psi_h}
=  0 + 2\,\Im\scp{L^2}{\psi_h\,\nabla\pot_h}{\nabla\psi_h}\,,
\end{split}
\end{equation*}
and estimate
\begin{equation*}
2 \,|\psi_h|_{H^1}\,\partial_t |\psi_h|_{H^1} =
\partial_t |\psi_h|_{H^1}^2 \leq 2\,\big| \scp{L^2}{\psi_h \nabla\pot_h}{\nabla\psi_h}\big|
\leq 2\,\norm{L^2}{\psi_h \nabla\pot_h}\,\big| \psi_h\big|_{H^1}\,.
\end{equation*}
This implies~\eqref{eq:H1-EBh-property},
\begin{equation*}
\partial_t |\psi_h|_{H^1} \leq \norm{L^2}{\psi_h \nabla\pot_h}
\quad \text{for} ~~ \psi_h = \ebh(t,w_h)\,u_h\,,
\end{equation*}
concluding the proof.
\end{enumerate}
\end{proof}

\subsection{Interpolation bounds and inverse estimates} \label{sec:interp_inverse}
In our convergence analysis we will refer to the following standard interpolation and inverse estimates.

\begin{theorem}
\label{thm:Brenner-Scott1}
Suppose $1<p < \infty$ and $m-d/p>0$.
Then, for $ 0 \leq s \leq m $ and $ u \in W_p^m $,
\begin{equation}
\label{eq:Brenner-Scott1}
\|u-\Ih\,u\|_{W_p^s}\leq C\,h^{m-s}\,|u|_{W_p^m}\,,
\end{equation}
where $C$ depends on $m$ and $d$.

Furthermore,
\begin{equation}
\label{eq:Brenner-Scott-infty}
\|u-\Ih\,u\|_{L^\infty} \leq C h^{m-d/p}\,|u|_{W^m_p}\,,
\end{equation}
where $C$ depends on $m$ and $d$.
\end{theorem}
\noindent This follows from our assumptions and~\cite[Theorem 4.4.20]{brennerscott02}.
\begin{theorem}
	\label{thm:Brenner-Scott-Ph-conv}
	Suppose that the boundary of $\Omega$ is such
    that~\eqref{eq:Omega-property} holds. Then,
	\begin{equation*}
	\| u -\Ph u \|_{L^2} \leq C\, h^{m} |u|_{H^m}\,.
	\end{equation*}
\end{theorem}
The proof relies on a duality argument and can be found in~\cite[Theorem 5.4.8]{brennerscott02}.

\begin{theorem}[Inverse estimate]
\label{thm:Brenner-Scottinv} Suppose that $0<h<1$. Then there exists $C$ such that
\begin{equation*}
\| u_h \|_{H^1} \leq C\,h^{-1}\,\| u_h \|_{L^2}
\end{equation*}
for all $u_h\in \nV^h$.
\end{theorem}
This follows from the remark of~\cite[Theorem 4.5.11]{brennerscott02}.
\subsection{Bounds involving $\DPh$}
At first we note the $H^1$-regularity property (see~\eqref{eq:Ph-Deltahi}),
\begin{equation}
\label{eq:Hmin2-est}
\| u_h \|_{H^1} = \| \DPh f \|_{H^1} = \|\Ph\,\DP f\|_{H^1}  \leq C\,\|\DP f\|_{H^1} \leq  C\,\| f \|_{H^{-1}}\,.
\end{equation}
The following estimate will be useful:
\begin{proposition}
\label{thm:Hmin1-est}
{\revred For $ f\in L^2 $ and $g\in H^1$,}
\begin{equation}
\label{eq:Hmin1-est}
\| f\,g \|_{H^{-1}} \leq C\,\| f \|_{L^2}\,\| g \|_{H^1}\,,
\end{equation}
where $C$ depends on $d$ and on $\Omega$.
\end{proposition}
\begin{proof}
We apply Cauchy-Schwarz and H{\"o}lder inequalities
and the Sobolev embedding of $H^1$ in $L^4$,
\begin{equation*}
\| f\,g \|_{H^{-1}}
= \sup_{\| v \|_{H^1} = 1} |(f\,g,v)_{L^2}|
\leq \sup_{\| v \|_{H^1} = 1} \|f\|_{L^2}\,\|g\,v\|_{L^2}
\leq \sup_{\| v \|_{H^1} = 1} \|f\|_{L^2}\,\|g\|_{L^4}\,\|v\|_{L^4}
\leq C\,\| f \|_{L^2}\,\| g \|_{H^1}\,,
\end{equation*}
completing the proof.
\end{proof}

\begin{proposition}
\label{thm:DPh-L2-est}
For $a\in L^2$ and $b,c \in H^1_0$,
\begin{equation*}
\|\DPh(a\, b)\, c\|_{L^2} \leq C\,\| a \|_{L^2} \,\|b\|_{H^1}\,\|c\|_{H^1}\,,
\end{equation*}
where $C$ depends on $d$ and on $\Omega$.
\end{proposition}
\begin{proof}
We use H{\"o}lder's inequality,
the Sobolev embedding of $H^1$ in $L^4$,
the estimate~\eqref{eq:Hmin2-est}, and Proposition~\ref{thm:Hmin1-est}:
\begin{align*}
\|\DPh(a\,b)\,c\|_{L^2}
&\leq \| \DPh(a\,b)\|_{L^4}\,\|c\|_{L^4}
\leq C\,\| \DPh(a\,b)\|_{H^1}\,\|c\|_{H^1} \\
&\leq C\,\| a\,b\|_{H^{-1}}\,\| c\|_{H^1} \leq C\,\| a\|_{L^2}\,\|b\|_{H^1}\,\|c\|_{H^1}\,,
\end{align*}
completing the proof.
\end{proof}

\begin{proposition}
\label{thm:DPh-H1-est}
For $ a,b,c \in H^1_{0} $,
\begin{equation*}
\|(\nabla\DPh (a,b))\,c\|_{L^2}
\leq  C\,\| a\|_{L^3}\, \| b\|_{L^3}\,\|c\|_{H^1}+ C\,h\,\| a\|_{L^6}\, \| b\|_{L^6}\,\|c\|_{H^1}\,,
\end{equation*}
with a constant $C$ depending on $d$ and $\Omega$.
\end{proposition}
\begin{proof}

We use H{\"o}lder's inequality, apply Theorem~\ref{thm:Brenner-Scott1},
and use the Sobolev embedding of $H^1$ in $L^6$,
\begin{align*}
\|(\nabla\DPh (a\,b))\,c\|_{L^2}
& \leq \|\nabla\DPh (a\,b)\|_{L^3}\,\|c\|_{L^6} \\
& \leq \|\nabla(\DPh-\DP)\,(a\,b)\|_{L^3}\,\| c \|_{L^6} +
       \| \nabla\DP (a\, b)\|_{L^3}\,\| c \|_{L^6}\\
& \leq C\,h\,|\DP (a\,b)|_{W^2_3}\,\| c\|_{L^6} +  \|\DP (a\,b)\|_{W^1_3}\,\| c \|_{L^6} \\
& \leq C\,h\,\| a\, b \|_{L^3}\,\| c \|_{L^6} + \| a\, b \|_{W^{-1}_3}\,\| c \|_{L^6}\\
& \leq C\,h\,\| a \|_{L^6}\,\| b\|_{L^6}\,\| c\|_{H^1} + C\,\| a\,b \|_{W^{-1}_3}\,\| {\revred c} \|_{H^1}\\
& \leq C\,h\,\| a \|_{L^6}\,\| b\|_{L^6}\,\| c \|_{H^1} + C\,\| a \|_{L^3}\,\| b \|_{L^3}\,\| c \|_{H^1} \,,
\end{align*}
where the last inequality follows from a duality argument and the H\"older inequality, $\| a\, b\|_{W^{-1}_3}\leq \| a \|_{L^3}\, \| b\|_{L^3}$.
\end{proof}

\begin{corollary}
\label{thm:DPh-H1-est-corollary}
For $a,b,c \in H^1_{0}$,
\begin{equation}
\label{eq:DPh-H1-est-corr}
\|(\nabla\DPh (a \, b))\,c\|_{L^2}
\leq  C\,(1+h) \| a\|_{H^1}\,\| b \|_{H^1}\,\|c\|_{H^1}\,.
\end{equation}
\end{corollary}
\begin{proof}
This follows from Proposition~\ref{thm:DPh-H1-est} and the Sobolev embeddings of $H^1$ in $L^3$ and $L^6$.
\end{proof}

\subsection{Conditional $ H^1$-stability of the evolution operator $ \ebh(t,\,\cdot\,) $}

\begin{proposition}
\label{thm:EBh-H1-est}
For $\phi,\xi \in\nV^h$,
the evolution operator $ \ebh(t,\,\cdot\,) $ defined in~\eqref{eq:def-ebh} satisfies
\begin{equation}
\label{eq:EBh-H1-est}
\| \ebhtee{\phi}\,\xi\|_{H^1} \leq \ee^{t\,C \| \phi\|_{H^1}^2} \norm{H^1}{\xi}\,,
\end{equation}
with a constant $C$ depending on $d$ and $\Omega$.
\end{proposition}
\begin{proof} Let $ \psi_h = \ebhtee{\phi}\,\xi $ and
	$ \pot_h = \Delta_h^{-1}(|\phi|^2) = \Delta_h^{-1}(\phi \cdot \ol{\phi}) $. According to~\eqref{eq:H1-EBh-property},
	\begin{equation*}
	\partial_t\,\norm{H^1}{\psi_h} = \partial_t |\psi_h|_{H^1}
	\leq \norm{L^2}{\psi_h\,\nabla\pot_h} \,.
	\end{equation*}
	From Corollary~\ref{thm:DPh-H1-est-corollary} we obtain
	\begin{align*}
	\partial_t\,\norm{H^1}{\psi_h} = \norm{L^2}{\psi_h\,\nabla\pot_h} & \leq C\,(1+h) \| \phi\|_{H^1}\,\| \phi \|_{H^1}\,\|\psi_h\|_{H^1}
	\end{align*}
	which entails~\eqref{eq:EBh-H1-est} for $h<1$.
\end{proof}

\subsection{$ H^m $-regularity of the semi-discrete splitting solution}
Here we show that an $H^m$-bound for the semi-discrete splitting solution
$\nnS{n} \ini$ defined in~\eqref{eq:Strang-splitting} depends linearly on the $H^m$-norm
of the initial value $ \psi_0 $ times an exponential function
depending on lower order Sobolev norms.
Hence for bounded times $n \t \leq T$, the $H^m$-norm of the semi-discrete splitting solution
will not behave worse than the $H^m$-norm of the initial value.

\begin{proposition}
	\label{thm:Hm-regularity}
	Let $m\in \mathbb{N}$. If $\ini\in H^m$ and
	\begin{equation*}
	\| \nnS{n}\ini \|_{H^{1}} \leq M_1 \quad \text{for all $n$~with }\, n\t\leq T,
	\end{equation*}
	then
	\begin{equation}
	\label{eq:Hm-reg-bound}
	\| \nnS{n}\ini \|_{H^m} \leq \ee^{L_m\, n \t}\,\| \ini\|_{H^m}\quad \text{for} ~~ n \t \leq T\,, \quad m \geq 2\,,
	\end{equation}
   	where $L_m$ depends on $M_1$ and on $\| \ini\|_{H^j}$
    for all $j<m$. The specific dependence is indicated in the proof.
\end{proposition}

\begin{proof}
    Since $\EA$ conserves the $H^m$-norm, we only consider the properties of the splitting operator $\EB$,
    which is the solution of (see~Sec.~\ref{sec:setting})
		\begin{equation}
        \begin{cases}
        ~\partial_\t \psi = -\,\ii\,\DP(|u|^2)\,\psi\,, \label{eq:Hm-reg-1}  \\
        ~\psi\big|_{t=0} = u\,.
        \end{cases}
		\end{equation}
	The basic idea is to bound the right-hand side of~\eqref{eq:Hm-reg-1} in the corresponding $H^m$-norm, using the following estimates.
	By the H\"older inequality and the Sobolev embeddings of $H^2$ in $L^\infty$ and $H^1$ in $L^4$, we have
	 \begin{align*}
	 \| u\, v\|_{H^m}\leq C\, \sum_{j=2}^{m}\| u\|_{H^j} \| {\revred v} \|_{H^{m-j+2}}\,,\quad m\geq 2\,.
	 \end{align*}
	 We further use the bound
	 \begin{align*}
	 \begin{aligned}
	 \| \DP(|u|^2)\|_{H^2} & \leq C\, \| |u|^2\|_{L^2} \leq C\, \|u \|_{H^1}^2\,,\\
	 \| \DP(|u|^2)\|_{H^m} & \leq C\, \| |u|^2\|_{H^{m-2}}\leq C\, \|u\|_{H^l}^2\,, \quad l=\max\{2,m-2\}\,, \quad \text{for}~ m \geq 3\,,
	 \end{aligned}
	 \end{align*}
	 and obtain
     \begin{subequations}
     \label{eq:DP-bounds}
	 \begin{align}
	 \| \DP(|u|^2)\, v\|_{H^2} &\leq C\, \|u \|_{H^1}^2 \| v\|_{H^2}\,, \label{eq:DP-bound-2}\\
	 \| \DP(|u|^2)\, v\|_{H^m} &\leq C\, \sum_{j=2}^{m-2} \| u \|_{H^{m-j}}^2 \| v\|_{H^j} + C\, \|u \|_{H^2}^2 \| v\|_{H^{m-1}} + C\, \| u \|_{H^1}^2  \| v\|_{H^m}
     \quad \text{for}\ m\geq 3\,.\label{eq:DP-bound-m}
	 \end{align}
     \end{subequations}
	For the proof of~\eqref{eq:Hm-reg-bound},
    we first proceed along the lines of the arguments from~\cite{lubich07},
    where the result was shown for $m=2$ and then extend the result for $m\in\{3,4,5\}$.
    For higher values of $m$ the proof works analogously but becomes technically more and more involved.
\begin{description}	
\item[$ \bullet $ ${m=2}\,$] From~\eqref{eq:Hm-reg-1} we obtain an integral inequality in the $H^2$-norm using the bound~\eqref{eq:DP-bound-2},
	\begin{equation*}
	\| \psi(\t)\|_{H^2} \leq \| u \|_{H^2} + \int_0^\t C\, M_1^2\,\| \psi(\s) \|_{H^2}\,\dd \s\,,
	\end{equation*}
	where $\| u \|_{H^1} \leq M_1$.
	
	By a Gronwall argument it follows that
	\begin{equation*}
	\| \psi (\t)\|_{H^2} \leq \ee^{\t\,C\, M_1^2}\,\| u\|_{H^2}\,,
	\end{equation*}
	and thus $\| \nS\, \ini\|_{H^2}\leq \ee^{\t\,C\, M_1^2}\,\| \ini\|_{H^2}$. Iterative application yields
	\begin{align}
	\label{eq:Hm-reg-m2bound}
	\| \nnS{n} \ini\|_{H^2}\leq \ee^{n\t\,C\, M_1^2}\,\| \ini\|_{H^2}
	\end{align}
	and the constant $L_2$ reads $L_2= C\, M_1^2$.
	
\item[$ \bullet $ $ m=3\,$] Again, we bound the right-hand side of~\eqref{eq:Hm-reg-1}. By~\eqref{eq:DP-bound-m}, we obtain for $m=3$,
\begin{align*}
\| \DP(|u|^2)\, \psi(\t)\|_{H^3} &\leq C\, \| u\|_{H^1}^2\| \psi(\t)\|_{H^3} + C\, \| u\|_{H^2}^2\| \psi(\t)\|_{H^2}\,.
\end{align*}
Hence,
\begin{align*}
\| \psi (\t)\|_{H^3} & \leq \| u\|_{H^3}
+ \int_0^\t \!\big( C\, \| u\|_{H^1}^2\| \psi(\sigma)\|_{H^3} + C\, \| u\|_{H^2}^2\| \psi(\s)\|_{H^2} \big)\,\dd\sigma\,,
\end{align*}
and by a Gronwall argument it follows that
\begin{align*}
\| \psi (\t)\|_{H^3} \leq \ee^{\t\,C\, M_1^2}\,\big(\| u\|_{H^3} + \t\, \sup_{0\leq \sigma \leq \t}\,C\, \| u\|_{H^2}^2\| \psi(\sigma)\|_{H^2}\big)\,.
\end{align*}
Setting $\psi=\nnS{n} \ini$, $u=\nnS{n-1} \ini$, and inserting the bound~\eqref{eq:Hm-reg-m2bound} for the $H^2$-norms, we conclude that
\begin{align*}
\| \nnS{n} \ini\|_{H^3}&\leq \ee^{\t\,C\, M_1^2}\,\| \nnS{n-1} \ini\|_{H^3} + \t \,C\, \ee^{\t\,C\, M_1^2}\, \ee^{2(n-1)\t\,C\, M_1^2}\,\| \ini\|_{H^2}^2\, \ee^{n\t\,C\, M_1^2}\,\| \ini\|_{H^2}\\
& \leq \ee^{\t\,C\, M_1^2}\,\|  \nnS{n-1} \ini\|_{H^3} + \t \,C\, \ee^{3 n\t\,C\, M_1^2}\,\| \ini\|_{H^2}^3 \\
& \leq \ee^{n\t\, C\, M_1^2}\, \| \ini\|_{H^3} + \sum_{i=0}^{n-1} \t \,C\, \ee^{3 n\t\,C\, M_1^2}\,\| \ini\|_{H^2}^3 \\
& \leq \ee^{n\t\, C\, M_1^2}\, \| \ini\|_{H^3} + n \t \,C\, \ee^{3 n\t\,C\, M_1^2}\,\| \ini\|_{H^2}^2\, \|\ini\|_{H^3} \\
& \leq \ee^{3n\t\, C\, M_1^2}\,  \big( 1+ n \t \,C\, \| \ini\|_{H^2}^2\big) \| \ini\|_{H^3}\\
& \leq \ee^{3n\t\, C\, M_1^2}\,  \ee^{n\t\, C\, \| \ini\|_{H^2}^2}\| \ini\|_{H^3}\\
& \leq \ee^{n\t\, C\, (3\,M_1^2+\| \ini\|_{H^2}^2)}\,\| \ini\|_{H^3}\,.
\end{align*}
\item[$ \bullet $ $m=4\,$] From~\eqref{eq:DP-bound-m} and~\eqref{eq:Hm-reg-1} we obtain
\begin{align*}
\| \psi (\t)\|_{H^4} & \leq \| u\|_{H^4} + \int_0^\t \big( C\, \| u\|_{H^1}^2\| \psi(\sigma)\|_{H^4} + C\, \| u\|_{H^2}^2\| \psi(\sigma)\|_{H^3} +  C\, \| u\|_{H^2}^2\| \psi(\sigma)\|_{H^2} \big)\, \dd\sigma\,,
\end{align*}
hence
\begin{align*}
\| \psi (\t)\|_{H^4} \leq \ee^{\t\,C\, M_1^2}\,
\Big(\| u\|_{H^4} + \t\,C\,\| u\|_{H^2}^2 \sup_{0\leq \sigma \leq \t}\,\big( \| \psi(\s)\|_{H^3} + \| \psi(\s)\|_{H^2}\big) \Big)\,.
\end{align*}
Setting $\psi=\nnS{n} \ini$ and $u=\nnS{n-1} \ini$ and following the argument for $m=3$, we conclude that
\begin{align*}
\| \nnS{n} \ini\|_{H^4}&\leq \ee^{\t\,C\, M_1^2}\,\|  \nnS{n-1} \ini\|_{H^4} \\
& \qquad + \t \,C\,  \ee^{2n\t\,C\, M_1^2}\,\| \ini\|_{H^2}^2\,\big( \ee^{n\t\, C\, (3\,M_1^2+\| \ini\|_{H^2}^2)}\,\| \ini\|_{H^3} +  \ee^{n\t\,C\, M_1^2}\,\| \ini\|_{H^2}\big) \\
& \leq \ee^{\t\,C\, M_1^2}\,\|  \nnS{n-1} \ini\|_{H^4} + \t \,C\, \ee^{n\t\, C\, (5\,M_1^2+\| \ini\|_{H^2}^2)}\, \| \ini\|_{H^2}^2\,\| \ini\|_{H^3} \\
& \leq \ee^{n\t\, C\, M_1^2}\, \| \ini\|_{H^4} + n \t \,C\, \ee^{n\t\, C\, (5\,M_1^2+\| \ini\|_{H^2}^2)}\, \| \ini\|_{H^2}^2\,\| \ini\|_{H^4} \\
& \leq \ee^{n\t\, C\, (5\,M_1^2+\| \ini\|_{H^2}^2)}\,  \big( 1+ n \t \,C\, \| \ini\|_{H^2}^2\big) \| \ini\|_{H^4}\\
& \leq\ee^{n\t\, C\, (5\,M_1^2+2\,\| \ini\|_{H^2}^2)}\,\| \ini\|_{H^4}\,.
\end{align*}
\item[$ \bullet $ $m=5\,$] In a similar way as before we obtain
\begin{align*}
\| \psi (\t)\|_{H^5} & \leq \| u\|_{H^5} + C\,\int_0^\t   \| u\|_{H^1}^2\| \psi(\sigma)\|_{H^5} + \| u\|_{H^2}^2(\| \psi(\sigma)\|_{H^4} +\| \psi(\sigma)\|_{H^3}) +   \| u\|_{H^3}^2\| \psi(\sigma)\|_{H^2}\, \dd\sigma
\end{align*}
and thus
\begin{align*}
\| \nnS{n} \ini\|_{H^5}&\leq \ee^{\t\,C\, M_1^2}\,\|  \nnS{n-1} \ini\|_{H^5} + \t \,C\,  \ee^{2n\t\,C\, M_1^2}\,\| \ini\|_{H^2}^2\,\big( \ee^{n\t\, C\, (5\,M_1^2+2\,\| \ini\|_{H^2}^2)}\,\| \ini\|_{H^4} \\
& \qquad  +  \ee^{n\t\, C\, (3\,M_1^2+\| \ini\|_{H^2}^2)}\,\| \ini\|_{H^3}\big) + \t \, C\, \ee^{2 n\t\, C\, (3\,M_1^2+\| \ini\|_{H^2}^2)} \| \ini \|_{H^3}^2 \, \ee^{n\t\,C\, M_1^2}\,\| \ini\|_{H^2}\\
& \leq \ee^{\t\,C\, M_1^2}\,\|  \nnS{n-1} \ini\|_{H^5} + \t \,C\, \ee^{n\t\, C\, (7\,M_1^2+2\| \ini\|_{H^2}^2)}\, \| \ini\|_{H^2}\,\| \ini\|_{H^3}\, \|\ini\|_{H^4} \\
& \leq \ee^{n\t\, C\, (7\,M_1^2+2\| \ini\|_{H^2}^2)}\,  \big( 1+ n \t \,C\, \| \ini\|_{H^3}^2\big) \| \ini\|_{H^5}\\
& \leq\ee^{n\t\, C\, (7\,M_1^2+2\,\| \ini\|_{H^2}^2+ \|\ini\|_{H^3}^2)}\,\| \ini\|_{H^5}\,.
\end{align*}
\end{description}
\end{proof}

}

%

\bibliographystyle{plain}

\end{document}